\documentclass[leqno]{article}

\usepackage[frenchb,english]{babel}
\usepackage[utf8]{inputenc}
\usepackage{amsmath}
\usepackage{amssymb}
\usepackage{amsfonts}
\usepackage{enumerate}
\usepackage{vmargin}
\usepackage[all]{xy}
\usepackage{mathrsfs}
\usepackage{mathtools}
\usepackage{lmodern}
\usepackage{slashed}
\usepackage[colorlinks=true,linkcolor=blue,pagebackref=true]{hyperref}%
\setmarginsrb{3cm}{3cm}{3.5cm}{3cm}{0cm}{0cm}{1.5cm}{3cm}
\usepackage{comment}

\usepackage{accents}

\newlength{\dhatheight}

\newcommand{\A}{\ensuremath{\mathcal{A}}}

\newcommand{\B}{\mathrm{B}} 
\let\H\relax 
\newcommand{\H}{\mathrm{H}}
\newcommand{\I}{\mathrm{I}} 
\newcommand{\K}{\ensuremath{\mathbb{K}}}

\let\L\relax 
\newcommand{\L}{\mathrm{L}}
\let\O\relax 
\newcommand{\O}{\ensuremath{\mathbb{O}}}

\newcommand{\M}{\mathrm{M}}

\let\cal\relax
\newcommand{\cal}{\mathcal}

\newcommand{\Zc}{\mathrm{Z}}
\newcommand{\R}{\ensuremath{\mathbb{R}}}

\newcommand{\Id}{\mathrm{Id}}

\renewcommand{\leq}{\ensuremath{\leqslant}}
\renewcommand{\geq}{\ensuremath{\geqslant}}
\newcommand{\qed}{\hfill \vrule height6pt  width6pt depth0pt}

\newcommand{\norm}[1]{\left\Vert#1\right\Vert}

\newcommand{\xra}{\xrightarrow}
\newcommand{\co}{\colon}

\newcommand{\ot}{\otimes}
\newcommand{\ovl}{\overline}

\let\i\relax 
\newcommand{\i}{\mathrm{i}}

\newcommand{\ov}{\overset}
\newcommand{\sa}{\mathrm{sa}}
\newcommand{\JW}{\mathrm{JW}}

\newcommand{\JB}{\mathrm{JB}}

\newcommand{\JBW}{\mathrm{JBW}}

\newcommand{\epsi}{\varepsilon}
\renewcommand{\d}{\mathop{}\mathopen{}\mathrm{d}} 
\renewcommand{\d}{\mathop{}\mathopen{}\mathrm{d}}

\DeclareMathOperator{\Span}{span} 
\DeclareMathOperator{\tr}{Tr} 
\DeclareMathOperator{\Tr}{Tr} 

\newtheorem{thm}{Theorem}[section]

\newtheorem{prop}[thm]{Proposition}

\newtheorem{cor}[thm]{Corollary}
\newtheorem{lemma}[thm]{Lemma}

\newtheorem{remark}[thm]{Remark}
\newtheorem{example}[thm]{Example}

\newenvironment{proof}[1][]{\noindent {\it Proof #1} : }{\hbox{~}\qed
\smallskip
}

\usepackage{tocloft}
\numberwithin{equation}{section}
\usepackage[nottoc,notlot,notlof]{tocbibind}

\let\OLDthebibliography\thebibliography
\renewcommand\thebibliography[1]{
  \OLDthebibliography{#1}
  \setlength{\parskip}{0pt}
  \setlength{\itemsep}{0pt plus 0.3ex}
}

\newcommand\reallywidehat[1]{\arraycolsep=0pt\relax%
\begin{array}{c}
\stretchto{
  \scaleto{
    \scalerel*[\widthof{\ensuremath{#1}}]{\kern-.5pt\bigwedge\kern-.5pt}
    {\rule[-\textheight/2]{1ex}{\textheight}} 
  }{\textheight} %
}{0.5ex}\\           
#1\\                 
\rule{-1ex}{0ex}
\end{array}
}

\begin{document}
\selectlanguage{english}
\title{\bfseries{Spectral nonassociative $\mathrm{L}^p$-spaces for $\mathrm{JBW}^*$-algebras}}
\date{}
\author{\bfseries{C\'edric Arhancet}}%
\maketitle


\begin{abstract}
We complete the construction of tracial spectral nonassociative $\mathrm{L}^p$-spaces for general $\mathrm{JBW}^*$-algebras. More precisely, if $\mathcal{M}$ is a $\mathrm{JBW}^*$-algebra equipped with a normal finite faithful trace $\tau$ and $1 \leq p < \infty$, we prove that $\|x\|_{\mathrm{L}^p(\mathcal{M})} \overset{\mathrm{def}}{=} (\tau[(x^* \circ x)^{\frac p2}])^{\frac1p}$, where $x \in \mathcal{M}$, defines a norm on $\mathcal{M}$. This resolves the remaining exceptional case left open by the corresponding result for $\mathrm{JW}^*$-algebras. The main difficulty is the complexified Albert algebra $\mathrm{H}_3(\mathbb{O}_{\mathbb{C}})$, which admits no embedding into an associative operator algebra. To treat this case, we establish a Jordan analogue of the joint convexity of the Kiefer map $\mathrm{M}_n \times \mathrm{H}_n^{++} \to \mathrm{H}_n^{+}$, $(a,h) \mapsto a^*h^{-1}a$, where $\mathrm{H}_n$ is the space of Hermitian matrices, and a Jordan version of a variational formula of Carlen and Lieb. This provides a complex Banach space framework for Jordan-algebraic models arising in some generalized probabilistic theories. 
\end{abstract}


%
%
%
%
%
%

\makeatletter
 \renewcommand{\@makefntext}[1]{#1}
 \makeatother

{
  \hypersetup{linkcolor=blue}
 \tableofcontents
}

\section{Introduction}
\label{sec:Introduction}

Noncommutative $\L^p$-spaces occupy a central position at the interface between operator algebras, functional analysis and quantum information theory. In the tracial setting of a von Neumann algebra $\cal{M}$ equipped with a normal finite faithful trace $\tau$, their basic Dixmier norm \cite{Dix53} is the familiar spectral formula
\[
\norm{x}_{\L^{p,D}(\cal{M})}
\ov{\mathrm{def}}{=} \tau(|x|^p)^{\frac1p},\quad x \in \cal{M},
\]
where $|x| \ov{\mathrm{def}}{=} (x^*x)^{\frac12}$ is the modulus of the operator $x$ and where $1 \leq p < \infty$. This definition is compatible with complex interpolation \cite{BeL76} of the couple $(\cal{M},\cal{M}_*)$, where we use the canonical embedding $\cal{M} \to \cal{M}_*$, $x \mapsto \tau(x\,\cdot)$ of $\cal{M}$ into its predual $\cal{M}_*$. More precisely, we have an isometric isomorphism $\L^{p,D}(\cal{M}) = (\cal{M},\cal{M}_*)_{\frac{1}{p}}$. Modern operator-algebraic methods in quantum information theory make extensive use of these spaces. These are powerful tools for studying quantum channels, R\'enyi divergences and quantum Markov dynamics.

The origin of Jordan algebras is directly tied to quantum mechanics. Jordan introduced in \cite[(4) p.~286]{Jor33} the symmetrized product
\begin{equation}
\label{Jordan-product}
x \circ y 
=\frac12(xy+yx)
\end{equation}
in order to keep the product of observables within the real vector space of observables. Jordan observed in \cite[(12), p.~288]{Jor33} that, although the product is not associative, it satisfies the weaker associativity relation now referred to as the Jordan identity:
\begin{equation}
\label{Jordan-identity}
x \circ (y \circ x^2)
=(x \circ y)\circ x^2.
\end{equation}
The subsequent work of Jordan, von Neumann and Wigner \cite{JNW34} led to the classification of finite-dimensional formally real\footnote{\thefootnote. A Jordan algebra $\A$ with product $\circ$ is called formally real if $a_1 \circ a_1 + \cdots +a_n \circ a_n =0$ implies $a_1 = \cdots = a_n = 0$ for any $a_1, \ldots, a_n \in \A$.} Jordan algebras. Besides the selfadjoint parts of real, complex and quaternionic matrix algebras and the spin factors, this classification contains one genuinely exceptional object, the Albert algebra
$$
\H_3(\O)
\ov{\mathrm{def}}{=} \left\{\begin{bmatrix}
   a  & \alpha & \beta  \\
   \ovl{\alpha}  & b & \gamma	\\
    \ovl{\beta} & \ovl{\gamma} &  c  \\
\end{bmatrix}: \alpha,\beta,\gamma \in \O, a,b,c \in \R \right\}
$$ 
of Hermitian $3 \times 3$ matrices with entries in the octonion algebra $\O$ equipped with the Jordan product $\H_3(\O) \times \H_3(\O) \to \H_3(\O)$, $(x,y) \mapsto x \circ y$. Recall that $\O$ is an eight-dimensional nonassociative real division algebra and that the Albert algebra $\H_3(\O)$ is a 27-dimensional unital formally real Jordan algebra, see \cite[Proposition 2.9.2 p.~69]{HOS84} and \cite{Bae02}. It is named after Albert, who proved in \cite{Alb34} that it is exceptional, that is, it is not isomorphic to a Jordan subalgebra of an associative algebra endowed with the Jordan product \eqref{Jordan-product}.

In finite dimensions, the class of formally real Jordan algebras agrees with that of finite-dimensional $\JBW$-algebras. Recall that a $\JBW$-algebra is a complete normed real Jordan algebra satisfying specific geometric norm axioms (see \eqref{def-JB}), which admits a predual. Hence, the Jordan algebras considered in the seminal work of Jordan, von Neumann, and Wigner naturally fall within the theory of $\JBW$-algebras and, after complexification, within the framework of $\JBW^*$-algebras. The class of $\JBW^*$-algebras, introduced by Edwards in \cite{Edw80}, includes in particular the exceptional factor 
$$
\H_3(\O_{\mathbb{C}})
\ov{\mathrm{def}}{=} \left\{\begin{bmatrix}
   a  & \alpha & \beta  \\
   \ovl{\alpha}  & b & \gamma	\\
    \ovl{\beta} & \ovl{\gamma} &  c  \\
\end{bmatrix}: \alpha,\beta,\gamma \in \O_{\mathbb{C}}, a,b,c \in \mathbb{C} \right\}
$$ 
of Hermitian $3 \times 3$ matrices with entries in the algebra $\O_{\mathbb{C}}\ov{\mathrm{def}}{=} \O \ot_{\mathbb{R}} \mathbb{C}$ of complex octonions equipped with the Jordan product $\H_3(\O_{\mathbb{C}}) \times \H_3(\O_{\mathbb{C}}) \to \H_3(\O_{\mathbb{C}})$, $(x,y) \mapsto x \circ y$. 
Recall that a $\JW^*$-algebra is a weak*-closed Jordan-$*$-subalgebra of a von Neumann algebra, and hence constitutes a particular class of $\JBW^*$-algebras. Moreover, every von Neumann algebra $\cal{M}$, equipped with the Jordan product \eqref{Jordan-product}, is itself a $\JW^*$-algebra. Thus von Neumann algebras form an important subclass of both $\JW^*$-algebras and $\JBW^*$-algebras.

It is well known that the selfadjoint parts of $\JBW^*$-algebras and $\JW^*$-algebras are exactly the $\JBW$-algebras and the $\JW$-algebras, respectively. The $\JW$-algebras are real linear spaces of selfadjoint operators which are weak* closed and closed under the Jordan product $\circ$. The theory of $\JW$-algebras goes back to Topping \cite{Top65} and was subsequently developed by several authors. For more background on these classes of Jordan algebras, we refer to the monographs \cite{AlS03}, \cite{ARU97}, \cite{CGRP14}, \cite{CGRP18}, \cite{Chu12}, and \cite{HOS84}, as well as to the important paper \cite{Sto66} and the references therein. 

%
The first systematic theory of tracial nonassociative $\L^p$-spaces was developed by Iochum in \cite{Ioc86} (see also \cite{Ioc84}). If $1 \leq p < \infty$ and if $\cal{A}$ is a $\JBW$-algebra, with product $\circ$, equipped with a normal finite faithful trace $\tau$, Iochum introduced the \textit{real} Banach space $\L^{p,I}(\cal{A})$, defined as the completion of $\cal{A}$ for the norm
\begin{equation}
\label{Iochum-introduction}
\norm{x}_{\L^{p,I}(\cal{A})}
\ov{\mathrm{def}}{=}
\big(\tau(|x|^p)\big)^{\frac1p}, \quad x \in \cal{A},
\end{equation}
where $|x| \ov{\mathrm{def}}{=} (x \circ x)^{\frac12}$. In this context, a trace on a $\JBW$-algebra $\cal{A}$ is a positive linear functional $\tau$ satisfying
\begin{equation*}
\tau(x \circ (y \circ z))
=\tau((x \circ y) \circ z), \quad x,y,z \in \cal{A}.
\end{equation*}
He proved the analogues of the basic classical properties of $\L^p$-spaces, including H\"older's inequalities, duality and uniform convexity. An important feature of his approach is that it is global and direct. It does not rely on reducing the problem, through classification theorems, to associative models.


Modern operator-algebraic methods in quantum information theory make extensive use of complex numbers. For instance, in the associative setting, complex interpolation is a central ingredient in the proof of some important results. If Jordan-algebraic probabilistic models are to support an analytic theory comparable to the usual quantum theory, it is natural to seek genuinely complex nonassociative $\L^p$-spaces. It is an important reason for passing from the real setting of $\JBW$-algebras to the complex setting of $\JBW^*$-algebras. 

This leads to the spectral construction considered in \cite{Arh24a,ArL26}. If $\cal M$ is a $\JBW^*$-algebra equipped with a normal finite faithful trace $\tau$, satisfying in particular
\begin{equation}
\label{associative-trace}
\tau(x\circ(y\circ z))
=\tau((x\circ y)\circ z), \quad x,y,z \in \cal{M},
\end{equation}
we define
\begin{equation}
\label{norm-Lp-intro}
\norm{x}_{\L^p(\cal{M})}
\ov{\mathrm{def}}{=}
\left(\tau\left[(x^*\circ x)^{\frac p2}\right]\right)^{\frac1p},
\quad x \in \cal{M},
\end{equation}
where $1 \leq p < \infty$. On the selfadjoint part $\cal A=\cal M_{\sa}$ this reduces precisely to the norm \eqref{Iochum-introduction} considered by Iochum. However, for a general element, the situation changes substantially. In \cite{ArL26}, the spectral expression \eqref{norm-Lp-intro} was proved to define a norm for every $\JW^*$-algebra $\cal{M}$ equipped with a normal finite faithful trace $\tau$. Moreover, the resulting spectral norm was compared with the norm of the complex interpolation space $\L^{p,A}(\cal{M}) \ov{\mathrm{def}}{=} (\cal M,\cal M_*)_{\frac1p}$ obtained from the compatible couple $(\cal M,\cal M_*)$, introduced in \cite{Arh24a}, yielding a striking equivalence of norms, with universal constants and hence a Banach-space isomorphism, rather than an isometry, between the corresponding spaces. More precisely, if $1 \leq p \leq 2$, it is established in \cite[Theorem 5.1]{ArL26} that
\begin{equation}
\label{equiv-1}
\norm{x}_{\L^{p,A}(\cal{M})}
\leq\norm{x}_{\L^p(\cal{M})}
\leq2^{\frac1p-\frac12}\norm{x}_{\L^{p,A}(\cal{M})},
\end{equation}
whereas, if $2 \leq p < \infty$, then
\begin{equation}
\label{equiv-2}
\norm{x}_{\L^p(\cal M)}
\leq\norm{x}_{\L^{p,A}(\cal M)}
\leq2^{\frac12-\frac1p}\norm{x}_{\L^p(\cal{M})}.
\end{equation}
where we use the endpoint convention $\L^{1,A}(\cal M)\ov{\mathrm{def}}{=} \cal M_*$. The constants are universal and optimal. Thus, unlike the associative theory of tracial von Neumann algebras, where the spectral and interpolation constructions yield the same norm, the Jordan product produces a genuine metric distinction between them while preserving a uniform equivalence depending only on $p$.

The setting of $\JW^*$-algebras leaves untouched precisely the exceptional part of the theory. By the structure theory of $\JBW$-algebras, the remaining obstruction is concentrated in algebras of the form $\L^\infty(\Omega,\H_3(\O_{\mathbb C}))$, where $\Omega$ is a finite measure space and where $\O_{\mathbb C}$ denotes the algebra of complex octonions. This led to the question of whether \eqref{norm-Lp-intro} is a norm on such algebras. The main purpose of the present paper is to answer this question positively and, as a consequence, to complete the spectral construction for arbitrary $\JBW^*$-algebras equipped with normal finite faithful trace. More precisely, we prove that if $\cal{M}$ is any $\JBW^*$-algebra equipped with a normal finite faithful trace and $1 \leq p < \infty$, then \eqref{norm-Lp-intro} defines a norm on $\cal{M}$. Thus the exceptional Albert component, despite having no associative realization, satisfies the same fundamental Minkowski property as the special Jordan components. 
The second purpose of the present paper is also to show that the equivalences \eqref{equiv-1} and \eqref{equiv-2} persist in the exceptional part of the theory.

\paragraph{Scope of the paper} 
A broader motivation for the present work comes from generalized probabilistic theories. In \cite[Section~9]{ArL26}, we explain why $\JBW$-algebras, and equivalently after complexification their associated $\JBW^*$-algebras, provide a natural framework for a large class of generalized probabilistic theories, see \cite{BGW15}, \cite{Bar07}, \cite{BaW11}, \cite{BaW14}, \cite{CDP11}, \cite{Cho19}, \cite{DL70}, \cite{HaW12}, \cite{Har01}, \cite{JaH14}, \cite{LRTF21}, \cite{LPW18}, \cite{Lam18}, \cite{Lud83}, \cite{Lud85}, \cite{Nie20}, \cite{Mul21}, \cite{Pla23}, \cite{Sha21}, \cite{Wil19}, \cite{Wil25} and \cite{WW21}.  
Several reconstruction results recover precisely such Jordan algebras from operational or information-theoretic assumptions, see, for instance, \cite{BUW23} and the references therein. Recent work goes beyond this structural correspondence and develops genuine information theory on Jordan models: for example, Stein's lemma and Jordan analogues of relative and R\'enyi entropies have been established for GPTs associated with Euclidean Jordan algebras in \cite{SAH25}. From this perspective, nonassociative $\L^p$-spaces are not merely auxiliary Banach spaces. They provide a natural candidate for the analytic layer underlying interpolation, norm estimates for transformations, entropic quantities and asymptotic information-theoretic arguments, much as noncommutative $\L^p$-spaces do in ordinary quantum information theory. The results of the present paper complete this analytic framework in the tracial setting by including the exceptional case. The spectral construction now produces genuine complex $\L^p$-spaces for every tracial $\JBW^*$-algebra, including $\H_3(\O_{\mathbb{C}})$, where no reduction to an associative operator algebra is available.

There is also a wider physical reason for taking nonstandard Jordan models seriously as mathematical laboratories. A substantial line of research in quantum foundations asks which parts of the standard complex Hilbert-space formalism are forced by experiment and which are contingent features of the usual formulation. This question is not purely axiomatic. For example, Renou et al.~\cite{RTWL21} proved that real and complex quantum theory make different predictions in network scenarios involving independent sources. The predicted separation was subsequently observed in quantum-network experiments \cite{LMW22}, ruling out the standard real-Hilbert-space formulation as a universal
physical theory. More recently, multipath and multiparticle tests have been proposed
to probe quaternionic and, more generally, hypercomplex extensions of quantum
mechanics \cite{SZP25}. Thus alternative scalar and algebraic structures remain experimentally meaningful probes of the foundations of quantum theory, even when a particular alternative is eventually excluded. Within the Jordan landscape, real and quaternionic quantum mechanics, spin factors and the exceptional Albert algebra arise on an equal footing at the single-system level, while their behavior under composition can be radically different: under natural assumptions on composites, an exceptional Jordan system cannot form a nontrivial composite with another nonclassical Jordan system \cite{BGW20}. This makes the Albert algebra a particularly useful boundary case for identifying which analytical features of quantum theory follow from order, spectral and Jordan structures alone, rather than from associativity. Moreover, exceptional Jordan and octonionic structures occur independently in mathematical physics, notably in the description of charge spaces, duality symmetries and BPS black holes in octonionic magical supergravity, see, for instance \cite{Bae02} and \cite{GuK24}. Although these appearances should not be conflated with a direct physical interpretation of $\H_3(\O_{\mathbb C})$ as a quantum system, they reinforce the broader point that exceptional Jordan structures are not isolated algebraic curiosities. Developing an intrinsic complex $\L^p$-theory for them therefore provides a concrete analytic tool for exploring the boundary between ordinary quantum mechanics and more general probabilistic or algebraic models.

We also refer to \cite{CGRP18}, \cite{EPV26}, \cite{Isi19}, \cite{LMO+26}, \cite{LuW26}, \cite{McC78}, \cite{McC04}, \cite{Upm85} and \cite{Upm87} for a description of some other applications of Jordan algebras to complex analysis in finite and infinite dimensions, harmonic analysis on homogeneous spaces, operator theory, projective geometry, quantum information theory and foundations of quantum mechanics.

\paragraph{Approach of the paper} The exceptional case requires a different argument from the one available for $\JW^*$-algebras. Let $\cal{A} \ov{\mathrm{def}}{=} \H_3(\O)$. If $x=a+\i b$ is an element of the $\JBW^*$-algebra $\cal{M} \ov{\mathrm{def}}{=} \H_3(\O_{\mathbb C})=\cal{A}+\i\cal{A}$, then we show that
\[
\norm{x}_{\L^p(\cal{M})}
=
\norm{a^2+b^2}_{\L^{\frac{p}{2},I}(\cal{A})}^{\frac12}.
\]
For any $p \geq 2$, the triangle inequality can be deduced directly from order inequalities in $\cal A$ and the monotonicity and Minkowski inequality of Iochum's $\L^{\frac{p}{2}}$-norm. The range $1 \leq p < 2$ is substantially more delicate, since $\frac{p}{2}<1$ and the latter argument is no longer available. To overcome this difficulty, we develop a variational approach inspired by matrix analysis. We first prove that the Jordan fractional map
\[
\cal{A} \times \cal{A}_{++} \to \cal{A}_+, (a,h) \mapsto U_a(h^{-1})
\]
is jointly convex for a $\JBW$-algebra $\cal{A}$. This can be viewed as the Jordan counterpart of the classical joint convexity \cite{Kie59} of the matrix function $\M_n \times \mathrm{H}_n^{++} \to \mathrm{H}_n^{+}$, $(a,h) \mapsto a^*h^{-1}a$. The proof uses the Jordan version of Kadison's inequality from \cite{RoY82} together with the fundamental formula for quadratic representations. We then establish a Jordan analogue of a  variational formula of Carlen and Lieb \cite{CaL08}:
\[
\tau(u^q)
=
\inf_{h\in\cal A_{++}}
\left\{
q\tau(u\circ h^{-1})+(1-q)\tau\big(h^{\frac q{1-q}}\big)
\right\},
\quad 0 < q < 1.
\]
These two ingredients imply the joint convexity of $(a,b) \mapsto \tau\big((a^2+b^2)^{\frac p2}\big)$, and its $p$-homogeneity then yields the desired Minkowski inequality. In this way, methods originating in the theory of matrix perspectives and trace inequalities survive in the genuinely exceptional Jordan setting without passing through an associative representation.

We next treat the measurable exceptional part. For $\cal{M} \ov{\mathrm{def}}{=} \L^\infty(\Omega,\H_3(\O_{\mathbb{C}}))$, we show that any normal finite faithful trace $\tau$ on $\cal{M}$ disintegrates against the unique normalized trace of the Albert factor. More precisely, there exists a finite measure $\nu$, equivalent to the original measure, such that
\[
\tau(f)
=\int_\Omega\tau_{\H_3(\O_{\mathbb C})}(f(\omega)) \d \nu(\omega).
\]
Consequently, for any $1 \leq p < \infty$, we have
\[
\norm{f}_{\L^p(\cal{M})}
=
\left(
\int_\Omega
\norm{f(\omega)}_{\L^p(\H_3(\O_{\mathbb C}))}^p \d\nu(\omega)
\right)^{\frac1p},
\]
and the resulting completion identifies canonically with the Bochner space $\L^p(\Omega,\nu,\L^p(\H_3(\O_{\mathbb{C}}))$. Combining this description with the decomposition of a general $\JBW$-algebra into its special and purely exceptional parts completes the proof of the general theorem.

\paragraph{Structure of the paper}
The paper is organized as follows. In Section~\ref{sec-preliminaries} we recall the necessary background on $\JBW$-algebras and $\JBW^*$-algebras, traces and Iochum's nonassociative $\L^p$-spaces. Section~\ref{sec-H3OC} is devoted to the complexified Albert algebra. We establish the required joint convexity and variational results and prove the triangle inequality for \eqref{norm-Lp-intro}. In Section~\ref{sec-Linfty}, we treat the case of the algebra $\L^\infty(\Omega,\H_3(\O_{\mathbb C}))$, identify the trace and the corresponding $\L^p$-norm fiberwise, and obtain a Bochner-space description. Finally, in Section~\ref{sec-JBW-star} we combine the exceptional case with the structure theory of $\JBW^*$-algebras to prove the general result. In Section~\ref{sec-comparison-interpolation-Albert}, we return to the comparison between the spectral and complex interpolation constructions. We show that, for the algebra $\L^\infty(\Omega,\H_3(\O_{\mathbb C}))$, the two norms satisfy exactly the same sharp comparison estimates as in the $\JW^*$-case. In particular, the exceptional component does not increase the universal distortion $2^{|\frac1p-\frac12|}$ between the two natural nonassociative $\L^p$-norms.

\section{Preliminaries}
\label{sec-preliminaries}

The purpose of this section is twofold. We fix terminology from Jordan operator algebra theory, and we isolate the few structural facts that will be repeatedly used in the paper.

\paragraph{Jordan algebras} A Jordan algebra $\A$ over a field $\K$ is a vector space $\A$ over $\K$ equipped with a (not necessarily associative) commutative bilinear product $\circ$ that satisfies the Jordan identity $x \circ (x^2 \circ y) = x^2\circ (x \circ y)$ of \eqref{Jordan-identity} for any $x,y \in \A$, see e.g.~\cite[p.~162]{CGRP14} or \cite[Definition 1.1 p.~3]{AlS03}. This means that the multiplication operators by $x$ and $x^2$ commute. 

Following \cite[Definition 1.5 p.~5]{AlS03} and \cite[3.1.4 p.~76]{HOS84}, a $\JB$-algebra is a Jordan algebra $\cal{A}$ over the scalar field $\R$ equipped with a complete norm satisfying the properties 
\begin{equation}
\label{def-JB}
\norm{x \circ y} \leq \norm{x} \norm{y}, \quad \|x^2\|=\norm{x}^2 
\quad \text{and} \quad \|x^2\| \leq \|x^2+y^2\|
\end{equation}
for any $x,y \in \A$. A $\JBW$-algebra is a $\JB$-algebra which is a dual Banach space \cite[p.~111]{HOS84}. In this case, the predual is unique.

A $\JB^*$-algebra \cite[p.~91]{HOS84} \cite[Definition 3.3.1 p.~345]{CGRP14} is a complex Banach space $\cal{M}$ which is a complex Jordan algebra equipped with an involution satisfying 
\begin{equation}
\label{def-JBstar}
\norm{x \circ y} \leq \norm{x} \norm{y},\quad  \norm{x^*}=\norm{x} 
\quad \text{and} \quad \norm{\{x,x^*,x\}}=\norm{x}^3
\end{equation}
for any $x,y \in \cal{M}$, where we use the Jordan triple product \cite[p.~25]{HOS84}\footnote{\thefootnote. We warn the reader that another definition is often used in some papers for the triple product:
\begin{equation}
\label{triple product} 
\{x,y,z\} 
 \ov{\mathrm{def}}{=} (x \circ y^*) \circ z + (y^* \circ z) \circ x - (x \circ z) \circ y^*.
\end{equation}}  
$$
\{x,y,z\}
\ov{\mathrm{def}}{=} (x \circ y) \circ z+(y \circ z) \circ x-(x \circ z) \circ y.
$$ 
As in \cite[p.~4]{CGRP18}, we say that a $\JB^*$-algebra which is a dual Banach space is a $\JBW^*$-algebra. By \cite[Lemma 4.1.7 p.~96]{HOS84}, any $\JBW$-algebra is unital.

An element $x$ in a $\JB^*$-algebra is called selfadjoint if $x^* = x$. According to \cite[Corollary 5.1.29 p.~9]{CGRP18}, the set $\cal{M}_\sa$ of selfadjoint elements of a $\JBW^*$-algebra $\cal{M}$ has a canonical structure of $\JBW$-algebra. Conversely, if $\A$ is a $\JBW$-algebra then by \cite[Corollary 5.1.41 p.~15]{CGRP18} there exists a unique $\JBW^*$-algebra $\cal{M}$ such that $\A$ is the selfadjoint part $\cal{M}_\sa$ of $\cal{M}$. 


An element $x$ of a $\JB^*$-algebra $\cal{M}$ is said to be positive \cite[p.~9]{CGRP18} \cite[p.~7]{AlS03} if $x \in \cal{M}_\sa$ and if we can write $x = y \circ y$ for some element $y \in \cal{M}_\sa$.

%


\paragraph{Centers and factors} Two elements $x$ and $y$ of a Jordan algebra $\A$ are said to operator commute \cite[p.~44]{HOS84} 
if for any $z \in \A$ we have $(x \circ z) \circ y= x \circ (z \circ y)$. The center $\Zc(\A)$ of $\A$ \cite[p.~45]{HOS84}  is the set of all elements of $\A$ which operator commute with all elements of $\A$. We say that an element $x \in \A$ is central if it belongs to the center. By \cite[Lemma 2.5.3 p.~45]{HOS84} and \cite[Proposition 2.36 p.~56]{AlS03}, the center $\Zc(\A)$ is an associative $\JBW$-subalgebra of $\A$. 

Following \cite[p.~115]{HOS84} and \cite[p.~348]{CGRP18}, if the center of a non-zero $\JBW$-algebra $\A$ only consists of scalar multiples of the identity, we say that $\A$ is a $\JBW$-factor. In this case, we say that the associated $\JBW^*$-algebra $\cal{M}$, such that $\cal{M}_\sa=\A$, is a $\JBW^*$-factor. We refer to \cite[Theorem 6.1.40 p.~362]{CGRP18} for a classification of $\JBW$-factors and $\JBW^*$-factors.

\paragraph{$\JW$-algebras} Recall that a (concrete) $\JW$-algebra \cite[p.~95]{HOS84} \cite[Definition 2.70 p.~70]{AlS03} 
 is a weak* closed Jordan subalgebra $\A$ of the selfadjoint part $\B(H)_\sa$ of the space $\B(H)$ of bounded operators on some complex Hilbert space $H$, that is a real linear space of selfadjoint operators which is closed for the weak* topology and closed under the Jordan product $\circ$. Note that a $\JW$-algebra is a $\JBW$-algebra by \cite[p.~95]{HOS84}. In this situation, by \cite[Proposition 1.49 p.~28]{AlS03}, two elements $x$ and $y$ of a $\JW$-algebra $\A$ operator commute if and only if $x$ and $y$ commute in $\B(H)$. A $\JW$-algebra which is a $\JBW$-factor is called a $\JW$-factor.

\paragraph{Projections} If $p$ is a projection (i.e.~$p \circ p=p$) of a $\JBW$-algebra $\A$, the smallest central projection $q$ such that $q \geq p$ is called the central cover of $p$ and denoted by $c(p)$, see \cite[Definition 2.38 p.~56]{AlS03}. We say that a projection $p$ of a $\JBW$-algebra $\A$ is abelian if the algebra $\{p,\A,p\}$ is associative \cite[p.~122]{HOS84} \cite[Definition 3.14 p.~85]{AlS03}. Two projections $p,q \in \A$ are said to be orthogonal if $p \circ q=0$ \cite[p.~45]{AlS03}. An element $p$ of a $\JBW^*$-algebra is said to be a projection if $p^* = p$ and $p \circ p = p$.

\paragraph{$\JW^*$-algebras} We define a $\JW^*$-algebra as a complex subspace of the space $\B(H)$ which is closed for the weak* topology and closed under the Jordan product $\circ$ and the involution, for some complex Hilbert space $H$.  
A $\JW^*$-algebra is a $\JBW^*$-algebra. The selfadjoint part of a $\JW^*$-algebra is a $\JW$-algebra. Conversely, if $\A$ is a $\JW$-algebra included in $\B(H)$ then the complexification $\A_\mathbb{C}=\A+\i \A$ is a $\JW^*$-algebra included in $\B(H)$. 


\paragraph{Traces} A trace on a $\JBW$-algebra $\A$ is a function $\tau$ defined on the subset $\A_+$ of positive
elements of $\A$ with values in $[0,+\infty]$ satisfying the following conditions:
\begin{enumerate}
\item $\tau(x+y)=\tau(x)+\tau(y)$ for any $x, y \in \A_+$,
\item $\tau(\lambda x)=\lambda\tau(x)$ for any $x \in \A_+$ and any $\lambda \geq 0$, where $0 \cdot (+\infty)=0$,
\item $\tau(s \circ x \circ s)=\tau(x)$ for any $x \in \A_+$ and any arbitrary symmetry $s$ of $\A$.
\end{enumerate}
Recall that a symmetry is an element $s$ such that $s \circ s=1$. The trace $\tau$ is said to be faithful if $\tau(x) > 0$ for all non-zero $x \in \A_+$, finite if $\tau(1) < + \infty$, semifinite if given any non-zero $x \in \A_+$ there is a non-zero $y \in \A_+$ such that $y \leq x$ with $\tau(y) < +\infty$. The trace $\tau$ is normal if for any increasing net $(x_i)$ of positive elements such that $x_i \to x$ where $x \in \A_+$, we have $\tau(x_i) \to \tau(x)$. We refer to \cite{AyA85}, \cite{Ayu82}, \cite{Ayu92}, \cite{Kin83} and \cite{PeS82} for more information on traces on $\JBW$-algebras.

Every finite trace on a $\JBW$-algebra $\A$ can be extended by linearity to a linear functional on $\A$. Thus a finite trace on a $\JBW$-algebra $\A$ can be seen as a positive linear functional $\tau$ satisfying the condition $\tau(s \circ x \circ s)=\tau(x)$ for all $x \in \A$ and all symmetries $s \in \A$. By \cite[Lemma 5.18 p.~147]{AlS03}, it is known that the last condition is equivalent to the formula 
\begin{equation}
\label{Def-trace}
\tau (x \circ (y \circ z)) 
= \tau((x \circ y) \circ z), \quad x,y,z \in \A.
\end{equation}
If $\tau$ is in addition normal and faithful, its complex-linear extension \cite[Corollary 5.1.41 p.~15]{CGRP18} provides a normal faithful positive linear functional on the associated $\JBW^*$-algebra $\cal{M}$ satisfying 
\begin{equation}
\label{trace}
\tau (x \circ(y \circ z))
=\tau((x \circ y) \circ z), \quad x,y,z \in \cal{M}.
\end{equation}
We say that such a map is a normal finite faithful trace on $\cal{M}$.
\section{Case of the $\JBW^*$-algebra $\H_3(\O_{\mathbb{C}})$}
\label{sec-H3OC}

We now address the main exceptional case. Let $\cal A=\H_3(\O)$ and let $\cal M=\H_3(\O_{\mathbb C})=\cal A+\i\cal A$ be its complexification, equipped with the normalized trace. Since $\cal M$ admits no realization as a Jordan subalgebra of an associative operator algebra, the embedding arguments used for $\JW^*$-algebras in \cite{ArL26} are no longer available. Our first observation is that the spectral norm of an element $x=a+\i b$ can nevertheless be expressed entirely in terms of Iochum's real nonassociative $\L^q$-spaces applied to the positive element $a^2+b^2$. This reduction will allow us to treat the range $p\geq2$ by order and Minkowski inequalities. The range $1\leq p<2$ requires new ingredients, developed below, based on joint convexity and a variational representation for fractional powers.

\begin{prop}
\label{prop-connection}
Suppose that $0 < p <\infty$. We let $q  \ov{\mathrm{def}}{=} \frac{p}{2}$. Let $x \in \cal{M}$. Consider its  unique decomposition $x = a + \i b$ with $a,b \in \cal{A}$. We have
\begin{equation}
\label{formule-utile}
\norm{x}_{\L^p(\cal{M})}
=\norm{a^2+b^2}_{\L^{q,I}(\cal A)}^{\frac{1}{2}}.
\end{equation}
\end{prop}

\begin{proof}
Since the Jordan product is commutative, we have
\begin{equation}
\label{Albert-xstarx-real-parts}
x^* \circ x
=(a+\i b)^* \circ (a+\i b)
=(a-\i b) \circ (a+\i b)
=a^2+b^2.
\end{equation}
Since $a^2+b^2\in\cal A_+$, we have $|a^2+b^2|=a^2+b^2$. Note that $\frac1p=\frac1{2q}$. Consequently, we obtain
\begin{align*}
\norm{x}_{\L^p(\cal M)}
&\ov{\eqref{norm-Lp-intro}}{=} \big(\tau[(x^*\circ x)^{\frac p2}]\big)^{\frac1p}
\ov{\eqref{Albert-xstarx-real-parts}}{=} \big(\tau[(a^2+b^2)^{\frac p2}]\big)^{\frac1p}
=\big(\tau[(a^2+b^2)^q]\big)^{\frac1{2q}}
\ov{\eqref{Iochum-introduction}}{=} \norm{a^2+b^2}_{\L^{q,I}(\cal A)}^{\frac12}.
\end{align*}
\end{proof}

Let $\cal{A}$ be a $\JBW$-algebra and let $\cal{M}=\cal{A}+\i \cal{A}$ be its complexification. We denote by $\cal{A}_{++}$ the cone of positive invertible elements of $\cal{A}$. For $a \in \cal{A}$, let
\[
U_a
\ov{\mathrm{def}}{=}
2L_a^2-L_{a^2} \co \cal{A} \to \cal{A}
\]
be the quadratic representation of $\cal{A}$, see \cite[(1.13) p.~9]{AlS03}, \cite[p.~35]{HOS84} and \cite[p.~121]{CGRP14}, where $L_a \co \cal{A} \to \cal{A}$, $b \mapsto a \circ b$ is the Jordan multiplication operator. More explicitly, we have
\begin{equation}
\label{quad-rep}
U_a(b)
=2 a \circ (a \circ b)-a^2 \circ b, \quad a,b \in \cal{A} .
\end{equation}

\begin{example} \normalfont
If $\cal{A}$ is a $\JW$-algebra, we have $U_a(b)=aba$ by \cite[p.~9]{AlS03} for any $a,b \in \cal{A}$.
\end{example}

Let $a \in \cal{A}$. By \cite[Theorem 1.25 p.~14]{AlS03}, the map $U_a \co \cal{A} \to \cal{A}$ is bounded and positive, i.e., $U_a(\cal{A}_+) \subset \cal{A}_+$. Moreover, according to \cite[Lemma 1.23 p.~13]{AlS03}, $a$ is an invertible element if and only if $U_a \co \cal{A} \to \cal{A}$ has a bounded inverse, and in this case the inverse map is $U_{a^{-1}} \co \cal{A} \to \cal{A}$.
Finally, by \cite[Proposition 3.4.15 p.~364]{CGRP14} we have the fundamental formula
\begin{equation}
\label{fundamental-formula-P-Albert}
U_{U_{a}(b)}
=U_{a}U_{b}U_a, \quad a,b \in \cal{A}.
\end{equation} 
Finally, we have
\begin{equation}
\label{Ua-de-1}
U_a(1)
=a^2.
\end{equation}

The following result is already known, see, e.g., \cite[Lemma 1 (v) p.~418]{Ioc86}. For the sake of completeness, we give the short computation.

\begin{lemma}
\label{lemma-trace-Ua}
Let $\cal{A}$ be a $\JBW$-algebra equipped with a finite trace. For any $b,c \in \cal{A}$, we have
\begin{equation}
\label{formula-Iochum}
\tau(U_b(c))
=\tau(b^2 \circ c).
\end{equation}
\end{lemma}

\begin{proof}
For any $b,c \in \cal{A}$, the tracial property gives
\begin{align*}
\MoveEqLeft
\tau(U_b(c))
\ov{\eqref{quad-rep}}{=} 2\tau\big(b\circ(b \circ c)\big)-\tau(b^2 \circ c)
\ov{\eqref{associative-trace}}{=} 2\tau\big(b^2 \circ c\big)-\tau(b^2 \circ c) 
= \tau(b^2 \circ c).
\end{align*}
\end{proof}

Now, we establish the following proposition. The first part can be viewed as a Jordan analogue of the joint convexity of the function $\M_n \times \mathrm{H}_n^{++} \to \mathrm{H}_n^{+}$, $(a,h) \mapsto a^*h^{-1}a$, proved in \cite[Lemma 3.2 p.~288]{Kie59} (see also \cite[p.~1]{Rus22}). It is also closely related to the nonassociative perspective functions on $\JB$-algebras introduced in \cite{WaW21}. A function $f \co C \times D \to \R$ is said to be jointly convex if
$$
f(t x_1 + (1-t)x_2, t y_1 + (1-t) y_2) 
\leq t f(x_1,y_1) + (1-t) f(x_2, y_2),\quad t \in [0,1],
$$
i.e., convex on the product $C \times D$ of two convex subsets of a vector space.

\begin{prop}
\label{prop-Albert-quadratic-fractional}
Let $\cal{A}$ be a $\JBW$-algebra. The map $\cal{A} \times \cal{A}_{++} \to \cal{A}_+$, $(a,h) \to U_a(h^{-1})$ is jointly convex with respect to the order of $\cal{A}$. Consequently, if $\tau$ is a normal finite trace on $\cal{A}$, then the map $\Phi \co \cal{A} \times \cal{A}_{++} \to [0,\infty)$, defined by
\begin{equation}
\label{def-de-Phi}
\Phi(a,h)
\ov{\mathrm{def}}{=}
\tau(a^2\circ h^{-1}), \quad a \in \cal{A}, h \in \cal{A}_{++},
\end{equation}
is jointly convex.
\end{prop}

\begin{proof}
Let $0 < \theta < 1$, let $a_1,a_2 \in \cal{A}$ and let $h_1,h_2 \in \cal{A}_{++}$. Put
\begin{equation}
\label{def-de-a-et-h}
a
\ov{\mathrm{def}}{=}
\theta a_1+(1-\theta)a_2
\quad\text{and}\quad
h
\ov{\mathrm{def}}{=}
\theta h_1+(1-\theta)h_2.
\end{equation}
Since $\cal{A}_{++}$ is convex, we have $h \in \cal{A}_{++}$. Consider the linear map $T \co \cal{A} \oplus \cal{A} \to \cal{A}$ defined by
\begin{equation}
\label{def-de-T}
T(z_1,z_2)
\ov{\mathrm{def}}{=}
U_{h^{-\frac{1}{2}}}\Big(\theta U_{h_1^{\frac{1}{2}}}(z_1)+(1-\theta)U_{h_2^{\frac{1}{2}}}(z_2)\Big).
\end{equation}
The map $T$ is positive. Moreover, we have
\begin{align*}
\MoveEqLeft
T(1,1)
\ov{\eqref{def-de-T}}{=}U_{h^{-\frac{1}{2}}}\Big(\theta U_{h_1^{\frac{1}{2}}}(1)+(1-\theta)U_{h_2^{\frac{1}{2}}}(1)\Big)
\ov{\eqref{Ua-de-1}}{=} U_{h^{-\frac{1}{2}}}\big(\theta h_1+(1-\theta)h_2\big) \\
&\ov{\eqref{def-de-a-et-h}}{=} U_{h^{-\frac{1}{2}}}(h) 
=U_{h^{-\frac12}}U_{h^{\frac12}}(1)
\ov{\eqref{Ua-de-1}}{=} 1.
\end{align*}
Thus $T$ is unital. By the Kadison inequality for positive unital maps between $\JB$-algebras \cite[Theorem 1.2 p.~365]{RoY82}, we have
\begin{equation}
\label{Kadison-Albert}
T(z_1,z_2)^2
\leq
T(z_1^2,z_2^2),
\qquad z_1,z_2 \in \cal{A}.
\end{equation}
Set $
z_1
\ov{\mathrm{def}}{=} U_{h_1^{-\frac{1}{2}}}(a_1)$ and $z_2
\ov{\mathrm{def}}{=}
U_{h_2^{-\frac{1}{2}}}(a_2)$. Since $U_{h_i^{\frac{1}{2}}}U_{h_i^{-\frac{1}{2}}}=\Id_\cal{A}$ for $i=1,2$, we have
\[
T(z_1,z_2)
\ov{\eqref{def-de-T}}{=} U_{h^{-\frac{1}{2}}}\Big(\theta U_{h_1^{\frac{1}{2}}}(z_1)+(1-\theta)U_{h_2^{\frac{1}{2}}}(z_2)\Big)
=U_{h^{-\frac{1}{2}}}(\theta a_1+(1-\theta)a_2)
\ov{\eqref{def-de-a-et-h}}{=} U_{h^{-\frac{1}{2}}}(a).
\]
Applying the positive map $U_{h^{\frac{1}{2}}} \co \cal{A} \to \cal{A}$ to \eqref{Kadison-Albert}, we obtain
\begin{align}
\MoveEqLeft
U_{h^{\frac{1}{2}}}\Big[\big(U_{h^{-\frac{1}{2}}}(a)\big)^2 \Big]
\leq
\theta U_{h_1^{\frac{1}{2}}}\Big[\big(U_{h_1^{-\frac{1}{2}}}(a_1)\big)^2\Big]
+(1-\theta)U_{h_2^{\frac{1}{2}}}\Big[\big(U_{h_2^{-\frac{1}{2}}}(a_2)\big)^2\Big].
\label{perspective-order-Albert}
\end{align}
We claim that for any $k \in \cal{A}_{++}$ and any $b \in \cal{A}$, we have
\begin{equation}
\label{identity-perspective-Albert}
U_{k^{\frac{1}{2}}}\Big[\big(U_{k^{-\frac{1}{2}}}(b)\big)^2\Big]
=U_{b}(k^{-1}).
\end{equation}
Indeed, using $c^2 \ov{\eqref{Ua-de-1}}{=}U_c(1)$ and the fundamental formula \eqref{fundamental-formula-P-Albert}, we obtain
\begin{align*}
U_{k^{\frac12}}\Big[\big(U_{k^{-\frac12}}(b)\big)^2\Big]
=U_{k^{\frac12}}U_{U_{k^{-\frac12}}(b)}(1)
\ov{\eqref{fundamental-formula-P-Albert}}{=} U_{k^{\frac12}}U_{k^{-\frac12}}U_bU_{k^{-\frac12}}(1)
\ov{\eqref{Ua-de-1}}{=} U_b(k^{-1}),
\end{align*}
where we used $U_{k^{\frac12}}U_{k^{-\frac12}}=\Id_{\cal A}$ in the last equality. Consequently, \eqref{perspective-order-Albert} becomes
\begin{equation}
\label{perspective-order-Albert-2}
U_{a}(h^{-1})
\leq \theta U_{a_1}(h_1^{-1}) + (1-\theta)U_{a_2}(h_2^{-1}).
\end{equation}
This proves the joint convexity of the map $(a,h) \mapsto U_a(h^{-1})$. 
Lemma \ref{lemma-trace-Ua} gives the equalities $\tau(U_{a}(h^{-1})) = \tau(a^2 \circ h^{-1}) = \Phi(a,h)$, $\tau(U_{a_1}(h_1^{-1})) = \tau(a_1^2 \circ h_1^{-1}) = \Phi(a_1,h_1)$ and $\tau(U_{a_2}(h_2^{-1})) = \tau(a_2^2 \circ h_2^{-1}) = \Phi(a_2,h_2)$. Consequently, applying $\tau$ to \eqref{perspective-order-Albert-2}, we conclude that
\[
\Phi(a,h)
\leq \theta \Phi(a_1,h_1) + (1-\theta) \Phi(a_2,h_2).
\]
Thus $\Phi$ is jointly convex.
\end{proof}

\paragraph{Spectral theory}
Let $\cal{M}$ be a finite-dimensional $\JB^*$-algebra. For our purposes, the key point is that selfadjoint elements admit a spectral decomposition into minimal projections, exactly as in the associative matrix setting. Indeed, by \cite[p.~44]{FaK94} (see also \cite[Theorem III.1.1 p.~43]{FaK94} and \cite[Proposition 2.2 p.~6]{HKP23}), any selfadjoint element $x \in \cal{M}$ admits a spectral decomposition of the form
\begin{equation}
\label{spectral-decomposition}
x
=\lambda_1(x) p_1+\dots+\lambda_n(x) p_n,
\end{equation}
where $p_1,\dots,p_n$ are mutually orthogonal minimal projections in $\cal{M}$ belonging to the unital subalgebra generated by $x$, with sum equal to $1$, and $\lambda_1(x),\dots,\lambda_n(x)$ are real numbers. 

\vspace{0.2cm}

We will use the following inequality.

\begin{lemma}
\label{lemma-Young-weighted}
Suppose $0 < q <1$. For any $\alpha \geq 0$ and any $\beta > 0$, we have
\begin{equation}
\label{scalar-Young-Albert}
\alpha^q
\leq q\alpha\beta^{-1}+(1-q)\beta^\frac{q}{1-q}.
\end{equation}
\end{lemma}

\begin{proof}
It suffices to apply the inequality $x^qy^{1-q} \leq qx+(1-q)y$ of \cite[p.~23]{Ste04} with $x=\alpha\beta^{-1}$ and $y=\beta^\frac{q}{1-q}$, since
\[
x^q y^{1-q}
=(\alpha\beta^{-1})^q(\beta^\frac{q}{1-q})^{1-q}
=\alpha^q.
\]
\end{proof}

The next result is a Jordan  analogue of the Carlen--Lieb variational formula \cite[Lemma 2.2 p.~115]{CaL08}. The proof below follows a reduction principle already used by Iochum in \cite[p.~431]{Ioc86}. Namely, given spectral decompositions $u=\sum_{i=1}^r\alpha_i c_i$ and $h=\sum_{j=1}^s \beta_j d_j$, the coefficients $m_{ij}\ov{\mathrm{def}}{=}\tau(c_i\circ d_j)$ are nonnegative and may be regarded as the masses of a positive measure on the product of the two finite spectra. This allows one to transfer scalar inequalities to the Jordan setting by applying them pointwise to each pair $(\alpha_i,\beta_j)$ and then summing with weights $m_{ij}$. 
This is a finite-dimensional instance of Iochum's observation that inequalities on measure spaces involving integrals of products of positive elements can be transferred to tracial nonassociative $\L^p$-spaces.

\begin{prop}
\label{prop-Albert-variational-fractional-power}
Let $\cal{A}$ be a finite-dimensional $\JBW$-algebra equipped with a normalized faithful trace $\tau$. Let $0 < q < 1$. For any $u \in \cal{A}_+$, we have
\begin{equation}
\label{variational-fractional-power-Albert}
\tau(u^q)
=
\inf_{h \in \cal{A}_{++}}
\left\{ q \tau(u \circ h^{-1})+(1-q)\tau\big(h^\frac{q}{1-q}\big) \right\}.
\end{equation}
\end{prop}

\begin{proof}
We let $\gamma \ov{\mathrm{def}}{=} \frac{q}{1-q}$. Consider some spectral decompositions
\begin{equation}
\label{spectral-decompositions}
u
=\sum_{i=1}^r \alpha_i c_i
\quad\text{and}\quad
h
=\sum_{j=1}^s \beta_j d_j
\end{equation}
of the positive elements $u$ and $h$, where $(c_i)_{i=1}^r$ and $(d_j)_{j=1}^s$ have sum 1, $\alpha_i \geq 0$, and $\beta_j>0$. For any integers $1 \leq i \leq r$ and $1 \leq j \leq s$, set
\begin{equation}
\label{def-mij}
m_{ij}
\ov{\mathrm{def}}{=} \tau(c_i\circ d_j).
\end{equation}
Since $c_i,d_j \in \cal A_+$, we have $m_{ij} \geq 0$ by \cite[Lemma 1 (v) p.~418]{Ioc86}. Moreover, we have
\begin{equation}
\label{un-calcul}
\sum_{j=1}^s m_{ij}
\ov{\eqref{def-mij}}{=} \tau\bigg(c_i \circ \sum_{j=1}^s d_j\bigg)
=\tau(c_i)
\quad\text{and}\quad
\sum_{i=1}^r m_{ij}
\ov{\eqref{def-mij}}{=} \tau\bigg(\sum_{i=1}^r c_i \circ d_j\bigg)
=\tau(d_j).
\end{equation}
Consequently, we have
\begin{equation}
\label{tau-uq-123}
\tau(u^q)
\ov{\eqref{spectral-decompositions}}{=} \tau\bigg(\bigg(\sum_{i=1}^r\alpha_i c_i\bigg)^q\bigg)
=\sum_{i=1}^r\alpha_i^q\tau(c_i)
\ov{\eqref{un-calcul}}{=} \sum_{i=1}^r\sum_{j=1}^s m_{ij}\alpha_i^q.
\end{equation}
Similarly, we have
\begin{align*}
\tau(u \circ h^{-1})
&\ov{\eqref{spectral-decompositions}}{=} \sum_{i=1}^r\sum_{j=1}^s\alpha_i\beta_j^{-1} \tau\big(  c_i \circ d_j\big)
\ov{\eqref{def-mij}}{=} \sum_{i=1}^r\sum_{j=1}^s m_{ij}\alpha_i\beta_j^{-1}
\end{align*}
and
\begin{align*}
\tau(h^\gamma)
\ov{\eqref{spectral-decompositions}}{=}\tau\bigg(\bigg(\sum_{j=1}^s \beta_j d_j\bigg)^\gamma\bigg)
&=\sum_{j=1}^s\beta_j^\gamma\tau(d_j)
\ov{\eqref{un-calcul}}{=} \sum_{i=1}^r\sum_{j=1}^s m_{ij}\beta_j^\gamma.
\end{align*}
Now, we use the inequality with $\alpha_i$ instead of $\alpha$ and $\beta_j$ instead of $\beta$, we obtain
$$
\alpha_i^q
\ov{\eqref{scalar-Young-Albert}}{\leq} q\alpha_i\beta_j^{-1}+(1-q)\beta_j^\frac{q}{1-q}, \quad 1 \leq i \leq r, 1 \leq j \leq s.
$$
Multiplying this inequality by $m_{ij}$ and summing over $i$ and $j$, we obtain
\[
\tau(u^q)
\leq q\tau(u\circ h^{-1})+(1-q)\tau(h^\gamma).
\]
Taking the infimum over $h\in\cal A_{++}$ gives
\[
\tau(u^q)
\leq
\inf_{h\in\cal A_{++}}
\left\{q\tau(u\circ h^{-1})+(1-q)\tau(h^\gamma)\right\}.
\]
For the reverse inequality, let $\epsi > 0$ and set $
h_\epsi
\ov{\mathrm{def}}{=}
(u+\epsi 1)^{1-q} \in \cal A_{++}$. Then
\[
h_\epsi^{-1}
=(u+\epsi1)^{q-1}
\quad\text{and}\quad
h_\epsi^\gamma
=(u+\epsi1)^q.
\]
Since $h_\epsi$ belongs to the associative Jordan subalgebra generated by $u$ and $1$ and since we have the spectral decomposition $u=\sum_{i=1}^r\alpha_i c_i$ (hence $u+\epsi1=\sum_{i=1}^r (\alpha_i+\epsi) c_i$), the spectral functional calculus gives
\[
h_\epsi^{-1}
=\sum_{i=1}^r(\alpha_i+\epsi)^{q-1}c_i
\quad\text{and}\quad
h_\epsi^\gamma
=\sum_{i=1}^r(\alpha_i+\epsi)^q c_i.
\]
Using $c_i \circ c_j = \delta_{ij}c_i$, we obtain
\[
u \circ h_\epsi^{-1}
=\bigg(\sum_{i=1}^r\alpha_i c_i\bigg) \circ\bigg(\sum_{i=1}^r(\alpha_i+\epsi)^{q-1}c_i\bigg)
=\sum_{i=1}^r\alpha_i(\alpha_i+\epsi)^{q-1}c_i.
\]
Hence, by linearity of $\tau$, we obtain
\begin{align*}
\MoveEqLeft
q\tau(u\circ h_\epsi^{-1})+(1-q)\tau(h_\epsi^\gamma)
=\sum_{i=1}^r
\left[
q\alpha_i(\alpha_i+\epsi)^{q-1}
+(1-q)(\alpha_i+\epsi)^q
\right]\tau(c_i).
\end{align*}
For each $1 \leq i \leq r$, the expression inside the brackets converges to $\alpha_i^q$ as $\epsi \to 0$. Hence
\[
q\tau(u\circ h_\epsi^{-1})+(1-q)\tau(h_\epsi^\gamma)
\xra[\epsi \to 0]{}
\sum_{i=1}^r\alpha_i^q\tau(c_i)
\ov{\eqref{tau-uq-123}}{=} \tau(u^q).
\]
This proves the reverse inequality and completes the proof.
\end{proof}

\begin{remark} \normalfont
\label{remark-Lieb}
In the associative finite-dimensional setting, the variational identity \eqref{variational-fractional-power-Albert} is a particular case, after a change of variables, of the variational formula of Carlen and Lieb \cite[Lemma 2.2 p.~115]{CaL08} \cite[Lemma 10.28 p.~349]{Car25}, which gives, for a positive matrix $A \in \M_n$, $0 < q < 1$, $r \ov{\mathrm{def}}{=} \frac{p}{q} > 1$, the variational formula
$$
\tr\left[(A^{\frac{p}{2}}BB^*A^{\frac{p}{2}})^{\frac{q}{p}}\right]
=\frac{1}{r} \inf_{X > 0} \tr\left[A^{\frac{p}{2}}B\frac{1}{X^{r-1}}B^*A^{\frac{p}{2}} +(r-1)X\right].
$$
Taking $p=1$, $A=u$ and $B=\I$ in their formula gives
\[
\tr(u^q)
=\frac1r \inf_{X > 0}
\left\{ \tr\big(uX^{1-r}\big)+(r-1)\tr(X) \right\}.
\]
Since $\frac1r=q$ and $1-r=-\frac{1-q}{q}$, this can be rewritten as
\[
\tr(u^q)
=\inf_{X > 0} \left\{q\tr\big(uX^{-\frac{1-q}{q}}\big)+(1-q)\tr(X) \right\}.
\]
Making the change of variables $
X
=h^{\frac{q}{1-q}}$, we obtain
\[
\tr(u^q)
=
\inf_{h > 0}
\left\{q\tr(uh^{-1})+(1-q)\Tr\big(h^{\frac{q}{1-q}}\big) \right\}.
\]
Thus \eqref{variational-fractional-power-Albert} can be viewed as a Jordan analogue of the Carlen--Lieb variational formula.
\end{remark}

\begin{remark} \normalfont
The case $q=\frac12$ of \eqref{variational-fractional-power-Albert} is closely related to the variational formula \cite[Proposition I.2 (1)]{FJR16}. Indeed, in this case $\gamma=\frac{q}{1-q}=1$ and \eqref{variational-fractional-power-Albert} becomes
\[
\tau(u^{\frac12})
\ov{\eqref{variational-fractional-power-Albert}}{=} \frac12\inf_{h\in\cal{A}_{++}}\left\{\tau(u\circ h^{-1})+\tau(h)\right\}.
\]
If $\cal{A}$ is associative, let $a$ be a positive invertible element and apply this formula with $u=a^2$. Then
\[
\tau(a)
=\frac12 \inf_{h > 0}\left\{\tau(a^2h^{-1})+\tau(h)\right\}.
\]
Making the change of variables $h=a^{\frac12} y a^{\frac12}$, which is a bijection of the positive invertible cone onto itself, and using the tracial property, we obtain
\[
\tau(h)=\tau(a^{\frac12}ya^{\frac12})
=\tau(ay)
\quad\text{and}\quad
\tau(a^2h^{-1})=\tau(a^2a^{-\frac12} y^{-1} a^{-\frac12})
=\tau(ay^{-1}).
\]
Consequently, we have
\[
\tau(a)
=\frac12\inf_{y>0}\left\{\tau(ay)+\tau(ay^{-1})\right\},
\]
which is exactly \cite[Proposition I.2 (1)]{FJR16}. 
\end{remark}

We will use the following lemma.

\begin{lemma}
\label{lemma-convexity}
Suppose that $1 \leq p < \infty$. Let $\cal{A}$ be a $\JBW$-algebra equipped with a normal finite faithful trace $\tau$. The map $\cal A_{+} \to [0,\infty)$, $h \mapsto \tau(h^p)$ is convex on $\cal{A}_+$. 
\end{lemma}

\begin{proof}
For any $h_1,h_2 \in \cal{A}_{+}$ and any $\theta \in [0,1]$, the triangle inequality in $\L^{p,I}(\cal A)$ and the scalar convexity of the function $t \mapsto t^p$ give
\begin{align*}
\MoveEqLeft
\tau\big((\theta h_1+(1-\theta)h_2)^p\big)
\ov{\eqref{Iochum-introduction}}{=}\norm{\theta h_1+(1-\theta)h_2}_{\L^{p,I}(\cal A)}^p 
\leq \big(\theta\norm{h_1}_{\L^{p,I}(\cal A)}+(1-\theta)\norm{h_2}_{\L^{p,I}(\cal A)}\big)^p \\
&\leq \theta\norm{h_1}_{\L^{p,I}(\cal A)}^p+(1-\theta)\norm{h_2}_{\L^{p,I}(\cal A)}^p
\ov{\eqref{Iochum-introduction}}{=} \theta\tau(h_1^p)+(1-\theta)\tau(h_2^p).
\end{align*}
\end{proof}

Now, we have collected the ingredients required to prove the triangle inequality in the exceptional case. The two ranges of $p$ exhibit rather different mechanisms. When $p\geq2$, Proposition~\ref{prop-connection} reduces the problem to the Banach norm of Iochum's $\L^{\frac{p}{2},I}(\cal{A})$-space, and a direct order estimate is sufficient. When $1 \leq p < 2$, the exponent $\frac{p}{2}$ lies below one, so this argument breaks down. In that range, the variational formula together with the joint convexity established above will imply the convexity of the $p$-th power of the candidate norm. Its $p$-homogeneity will then recover the triangle inequality. This yields the following main result for the complexified Albert algebra.

\begin{thm}
\label{thm-spectral-norm-complexified-Albert}
Let $\cal{M}=\H_3(\O_{\mathbb{C}})$ be equipped with its normalized trace $\tau$. Suppose that $1 \leq p < \infty$. Then \eqref{norm-Lp-intro} defines a norm on $\cal{M}$.
\end{thm}

\begin{proof}
By \cite[Proposition 3.1]{ArL26}, the expression of \eqref{norm-Lp-intro} is absolutely homogeneous and point-separating. It remains to prove the triangle inequality. We first suppose that $2 \leq p < \infty$. Put $
q
\ov{\mathrm{def}}{=} \frac p2 \geq 1$. Consider some elements $x=a+\i b$ and $y=c+\i d$ in the $\JBW^*$-algebra $\cal{M}$. We have 
\begin{equation}
\label{inter-sans-fin-39}
x+y
=a+c+\i(b+d).
\end{equation}
Set $
U
\ov{\mathrm{def}}{=}
a^2+b^2$ and $V
\ov{\mathrm{def}}{=}
c^2+d^2$. For any $t > 0$, we have
\begin{align}
\label{inequality-37}
\MoveEqLeft
0
\leq
\big(\sqrt{t}a-t^{-1/2}c\big)^2
+
\big(\sqrt{t}b-t^{-1/2}d\big)^2 
=t(a^2+b^2)+t^{-1}(c^2+d^2)-2(a \circ c+b \circ d) \\
&=tU+t^{-1}V-2(a \circ c+b \circ d). \nonumber
\end{align}
Consequently, we obtain
\begin{align}
\MoveEqLeft
\label{inequality-38}
(a+c)^2+(b+d)^2
=a^2+b^2+c^2+d^2+2(a \circ c+b \circ d) \\
&=U+V+2(a \circ c+b \circ d)
\ov{\eqref{inequality-37}}{\leq}
(1+t)U+(1+t^{-1})V. \nonumber        
\end{align}
Using the monotonicity and the triangle inequality for the norm $\norm{\cdot}_{\L^{q,I}(\cal{A})}$, we obtain
\begin{align*}
\MoveEqLeft
\norm{x+y}_{\L^p(\H_3(\O_{\mathbb{C}}))}^2
\ov{\eqref{inter-sans-fin-39}}{=} \norm{a+c+\i(b+d)}_{\L^p(\H_3(\O_{\mathbb{C}}))}^2
\ov{\eqref{formule-utile}}{=} \norm{(a+c)^2+(b+d)^2}_{\L^{q,I}(\cal{A})} \\
&\ov{\eqref{inequality-38}}{\leq} \norm{(1+t)U+(1+t^{-1})V}_{\L^{q,I}(\cal{A})} 
 \leq (1+t)\norm{U}_{\L^{q,I}(\cal{A})} + (1+t^{-1})\norm{V}_{\L^{q,I}(\cal{A})} \\
&=(1+t)\norm{a^2+b^2}_{\L^{q,I}(\cal{A})} + (1+t^{-1})\norm{c^2+d^2}_{\L^{q,I}(\cal{A})} \\
&= (1+t)\norm{x}_{\L^p(\H_3(\O_{\mathbb{C}}))}^2 + (1+t^{-1})\norm{y}_{\L^p(\H_3(\O_{\mathbb{C}}))}^2.
\end{align*}
Suppose first that $\norm{x}_{\L^p(\H_3(\O_{\mathbb{C}}))} > 0$ and $\norm{y}_{\L^p(\H_3(\O_{\mathbb{C}}))} > 0$. Choosing $
t
\ov{\mathrm{def}}{=}
\frac{\norm{y}_{\L^p(\H_3(\O_{\mathbb{C}}))}}{\norm{x}_{\L^p(\H_3(\O_{\mathbb{C}}))}}$ gives
\[
\norm{x+y}_{\L^p(\H_3(\O_{\mathbb{C}}))}^2
\leq
\big(\norm{x}_{\L^p(\H_3(\O_{\mathbb{C}}))}+\norm{y}_{\L^p(\H_3(\O_{\mathbb{C}}))}\big)^2.
\]
Taking the square roots, we conclude that
\[
\norm{x+y}_{\L^p(\H_3(\O_{\mathbb{C}}))}
\leq \norm{x}_{\L^p(\H_3(\O_{\mathbb{C}}))}+\norm{y}_{\L^p(\H_3(\O_{\mathbb{C}}))}.
\]
If one of $\norm{x}_{\L^p(\H_3(\O_{\mathbb{C}}))}$ or $\norm{y}_{\L^p(\H_3(\O_{\mathbb{C}}))}$ is zero, the same conclusion follows from the point-separating property. Thus the triangle inequality holds when $p \geq 2$.

Now, we suppose that $1 \leq p < 2$. Put $
q
\ov{\mathrm{def}}{=}
\frac p2$ and $
\gamma
\ov{\mathrm{def}}{=}
\frac{q}{1-q}
=
\frac{p}{2-p}$. Note that $\gamma \geq 1$. Define
\begin{equation}
\label{def-de-Fp}
F_p(a,b)
\ov{\mathrm{def}}{=}
\tau\big((a^2+b^2)^q\big),
\qquad a,b \in \cal{A}.
\end{equation}
By Proposition \ref{prop-Albert-variational-fractional-power}, we have
\begin{align}
F_p(a,b)
=
\inf_{h \in \cal{A}_{++}}
\Big\{ q\Phi(a,h) + q\Phi(b,h) + (1-q)\tau(h^\gamma) \Big\}.
\label{variational-Fp-Albert}
\end{align}
By Proposition \ref{prop-Albert-quadratic-fractional}, the maps $(a,h) \mapsto \Phi(a,h)$ and $(b,h) \mapsto \Phi(b,h)$ are jointly convex. By Lemma \ref{lemma-convexity}, the map $ h\mapsto \tau(h^\gamma)$ is convex on $\cal A_{++}$. Thus the function
\[
G(a,b,h)
\ov{\mathrm{def}}{=}
q\Phi(a,h)+q\Phi(b,h)+(1-q)\tau(h^\gamma)
\]
is jointly convex on $\cal A \oplus \cal A \oplus \cal A_{++}$. Now, we show that $F_p$ is convex (see also \cite[p.~87-88]{BoV04}). Let $(a_1,b_1),(a_2,b_2)\in\cal A\oplus\cal A$, let $0<\theta<1$, and let $\epsi>0$. Choose $h_1,h_2\in\cal A_{++}$ such that
\[
G(a_i,b_i,h_i)
\leq F_p(a_i,b_i)+\epsi,
\qquad i=1,2.
\]
Since $\cal A_{++}$ is convex, $h \ov{\mathrm{def}}{=}\theta h_1+(1-\theta)h_2$ belongs to $\cal A_{++}$. Hence
\begin{align*}
\MoveEqLeft
F_p\big(\theta(a_1,b_1)+(1-\theta)(a_2,b_2)\big)
\leq G\big(\theta a_1+(1-\theta)a_2,\theta b_1+(1-\theta)b_2,h\big) \\
&\leq \theta G(a_1,b_1,h_1)+(1-\theta)G(a_2,b_2,h_2)
\leq \theta F_p(a_1,b_1)+(1-\theta)F_p(a_2,b_2)+\epsi.
\end{align*}
Letting $\epsi \to 0$ proves that $F_p$ is convex.

It follows that the expression inside the braces in \eqref{variational-Fp-Albert} is jointly convex in $(a,b,h)$. Taking the infimum with respect to $h \in \cal{A}_{++}$ shows that $F_p$ is convex on $\cal{A} \oplus \cal{A}$. Furthermore, we have
\begin{equation}
\label{homo-Fp}
F_p(ta,tb)
\ov{\eqref{def-de-Fp}}{=} t^p F_p(a,b),
\quad t \geq 0.
\end{equation}
Let $x,y \in \cal{M}$. Put $
\alpha
\ov{\mathrm{def}}{=} \norm{x}_{\L^p(\H_3(\O_{\mathbb{C}}))}
$ and $
\beta
\ov{\mathrm{def}}{=} \norm{y}_{\L^p(\H_3(\O_{\mathbb{C}}))}$. Assume first that $\norm{x}_{\L^p(\H_3(\O_{\mathbb{C}}))} > 0$ and $\norm{y}_{\L^p(\H_3(\O_{\mathbb{C}}))} > 0$. Identifying $\cal{M}$ with the real vector space $\cal{A} \oplus \cal{A}$, we write $F_p(x) \ov{\mathrm{def}}{=}F_p(a,b)$ whenever $x=a+\i b$. If $x=a+\i b$, we have
\begin{equation}
\label{Fp-et norm}
F_p(x)
=F_p(a+\i b)
\ov{\eqref{def-de-Fp}}{=} \tau\big((a^2+b^2)^q\big)
\ov{\eqref{Iochum-introduction}}{=} \norm{a^2+b^2}_{\L^{q,I}(\cal A)}^{q}
\ov{\eqref{formule-utile}}{=} \norm{x}_{\L^p(\cal{M})}^p.
\end{equation}
Then the convexity and the $p$-homogeneity of $F_p$ give
\begin{align*}
\MoveEqLeft
\frac{\norm{x+y}_{\L^p(\H_3(\O_{\mathbb C}))}^p}{(\alpha+\beta)^p}
=\frac{F_p(x+y)}{(\alpha+\beta)^p} 
\ov{\eqref{homo-Fp}}{=} F_p\left(\frac{x+y}{\alpha+\beta}\right) \\
&=F_p\left(\frac{\alpha}{\alpha+\beta}\frac{x}{\alpha}
+\frac{\beta}{\alpha+\beta}\frac{y}{\beta}\right)
\leq\frac{\alpha}{\alpha+\beta}F_p\left(\frac{x}{\alpha}\right)
+\frac{\beta}{\alpha+\beta}F_p\left(\frac{y}{\beta}\right) \\
&\ov{\eqref{homo-Fp}}{=} \frac{\alpha}{\alpha+\beta}\frac{F_p(x)}{\alpha^p}+\frac{\beta}{\alpha+\beta}\frac{F_p(y)}{\beta^p} \\
&\ov{\eqref{Fp-et norm}}{=} \frac{\alpha}{\alpha+\beta}\frac{\norm{x}_{\L^p(\H_3(\O_{\mathbb C}))}^p}{\alpha^p}
+\frac{\beta}{\alpha+\beta}\frac{\norm{y}_{\L^p(\H_3(\O_{\mathbb C}))}^p}{\beta^p}
=\frac{\alpha}{\alpha+\beta}+\frac{\beta}{\alpha+\beta} =1.
\end{align*}
We obtain
\[
\norm{x+y}_{\L^p(\H_3(\O_{\mathbb{C}}))}
\leq\alpha+\beta
=\norm{x}_{\L^p(\H_3(\O_{\mathbb{C}}))}
+\norm{y}_{\L^p(\H_3(\O_{\mathbb{C}}))}.
\]
As before, the cases $\alpha=0$ or $\beta=0$ follow from the point-separating property. Thus the triangle inequality also holds for $1 \leq p < 2$.
\end{proof}

\section{Case of the $\JBW^*$-algebra $\L^\infty(\Omega,\H_3(\O_{\mathbb{C}}))$}
\label{sec-Linfty}

We next pass from the exceptional factor to the measurable exceptional algebras which occur in the structure theory of general $\JBW^*$-algebras. Having established the triangle inequality on the fibre $\H_3(\O_{\mathbb C})$, the remaining issue is to understand how an arbitrary normal finite faithful trace on $\L^\infty(\Omega,\H_3(\O_{\mathbb C}))$ decomposes over the center. The uniqueness of the normalized trace on the Albert factor implies that all the freedom is carried by a scalar measure on $\Omega$. We show below that the trace therefore disintegrates as the integral of the normalized fibre trace. As a consequence, the spectral nonassociative $\L^p$-space becomes a genuine Bochner space with values in the finite-dimensional Banach space $\L^p(\H_3(\O_{\mathbb C}))$.

\begin{prop}
\label{prop-traces-on-Linfty}
Let $(\Omega,\Sigma,\mu)$ be a localizable measure space. Assume that the $\JBW^*$-algebra $\cal{M}=\L^\infty(\Omega,\H_3(\O_{\mathbb{C}}))$ is equipped with a normal finite faithful trace $\tau$. Then there exists a finite measure $\nu$ on $\Omega$ equivalent to $\mu$ such that
\begin{equation}
\label{eq-trace-disintegration-Albert}
\tau(f)
=
\int_\Omega \tau_{\H_3(\O_{\mathbb{C}})}(f(\omega)) \d\nu(\omega),
\quad f \in \L^\infty(\Omega,\H_3(\O_{\mathbb{C}})),
\end{equation}
where $\tau_{\H_3(\O_{\mathbb{C}})}$ is the normalized trace on the $\JBW^*$-factor $\H_3(\O_{\mathbb{C}})$, and if $1 \leq p < \infty$ then
\begin{equation}
\label{eq-Lp-Bochner-Albert}
\norm{f}_{\L^p(\cal{M})}
=
\left( \int_\Omega \norm{f(\omega)}_{\L^p(\H_3(\O_{\mathbb{C}}))}^p \d\nu(\omega) \right)^{\frac1p}, \quad f \in \L^\infty(\Omega,\H_3(\O_{\mathbb{C}})).
\end{equation}
\end{prop}

\begin{proof}
For any measurable subset $B$ of $\Omega$, let $
\nu(B)
\ov{\mathrm{def}}{=}
\tau(\mathbf{1}_B1)$, where $\mathbf{1}_B1$ is the central projection of $\L^\infty(\Omega,\H_3(\O_{\mathbb{C}}))$ defined by
\[
(\mathbf{1}_B1)(\omega)
=
\mathbf{1}_B(\omega)1.
\]
Since the trace $\tau$ is normal, $\nu$ is a finite countably additive measure. Moreover, we have
\[
\nu(\Omega)
=
\tau(1)
<
\infty.
\]
Now, we show that $\nu$ and $\mu$ have the same null sets. Indeed, if $\nu(B)=0$, then $\tau(\mathbf{1}_B1)=0$. Since $\mathbf{1}_B1$ is positive and $\tau$ is faithful, we obtain $\mathbf{1}_B1=0$ in $\cal{M}$, hence $\mu(B)=0$. The converse is immediate. For a fixed measurable subset $B$ of $\Omega$, define the functional $\rho_B \co \H_3(\O_{\mathbb{C}}) \to \mathbb{C}$ by
\begin{equation}
\label{def-rho-B}
\rho_B(a)
\ov{\mathrm{def}}{=}
\tau(\mathbf{1}_Ba).
\end{equation}

Let $e\ov{\mathrm{def}}{=}\mathbf{1}_B1$. Since $e$ is central, we have
\[
(e\circ x)\circ y=e\circ(x\circ y), \qquad x,y\in\cal M.
\]
Moreover, $e\circ x=\mathbf{1}_Bx$. Hence, for any $a,b,c\in\H_3(\O_{\mathbb C})$,
\begin{align*}
\rho_B(a\circ(b\circ c))
&\ov{\eqref{def-rho-B}}{=}\tau\big(\mathbf{1}_B(a\circ(b\circ c))\big)
=\tau\big((\mathbf{1}_Ba)\circ(b\circ c)\big)\\
&\ov{\eqref{associative-trace}}{=}\tau\big(((\mathbf{1}_Ba)\circ b)\circ c\big)
=\tau\big(\mathbf{1}_B((a\circ b)\circ c)\big)
\ov{\eqref{def-rho-B}}{=}\rho_B((a\circ b)\circ c).
\end{align*}

So, the map $\rho_B$ is a positive finite trace on the $\JBW^*$-algebra $\H_3(\O_{\mathbb{C}})$. 
The restriction $\sigma_B\ov{\mathrm{def}}{=}\rho_B|_{\H_3(\O)}$ is a positive finite trace on $\H_3(\O)$. By \cite[Proposition 5.25 p.~152]{AlS03}, the algebra $\H_3(\O)$ admits a unique normalized trace, denoted by $\tau_{\H_3(\O)}$. If $\sigma_B(1)>0$, then $\frac{1}{\sigma_B(1)}\sigma_B$ is a normalized trace, and therefore
\[
\sigma_B
=\sigma_B(1)\tau_{\H_3(\O)}.
\]
If $\sigma_B(1)=0$, we have $\sigma_B=0$, so the same equality still holds. Since
\[
\sigma_B(1)=\rho_B(1)
\ov{\eqref{def-rho-B}}{=}\tau(\mathbf{1}_B1)
=\nu(B),
\]
and since $\rho_B$ and $\tau_{\H_3(\O_{\mathbb C})}$ are the complex-linear extensions of their restrictions to $\H_3(\O)$, we obtain
\[
\rho_B
=\nu(B)\tau_{\H_3(\O_{\mathbb C})}.
\]
Consequently, for any measurable subset $B$, we have
\begin{equation}
\label{eq-simple-fibre-trace}
\tau(\mathbf{1}_Ba)
=\nu(B)\tau_{\H_3(\O_{\mathbb C})}(a),
\quad a \in \H_3(\O_{\mathbb C}).
\end{equation}
Let $f = \sum_{k=1}^n \mathbf{1}_{B_k}a_k$ be a simple $\H_3(\O_{\mathbb{C}})$-valued function, where the sets $B_1,\ldots,B_n$ are pairwise disjoint. By \eqref{eq-simple-fibre-trace}, we obtain
\begin{align}
\MoveEqLeft
\label{tau-f}
\tau(f)
=\tau\bigg(\sum_{k=1}^n \mathbf{1}_{B_k}a_k\bigg)
=\sum_{k=1}^n \tau(\mathbf{1}_{B_k}a_k)
\ov{\eqref{eq-simple-fibre-trace}}{=}\sum_{k=1}^n \nu(B_k) \tau_{\H_3(\O_{\mathbb{C}})}(a_k) \\
&=\sum_{k=1}^n \tau_{\H_3(\O_{\mathbb{C}})}(a_k)\int_\Omega \mathbf{1}_{B_k}(\omega) \d\nu(\omega) 
=\int_\Omega \sum_{k=1}^n\tau_{\H_3(\O_{\mathbb{C}})}\big( \mathbf{1}_{B_k}(\omega)a_k\big) \d\nu(\omega) \nonumber \\
&=\int_\Omega \tau_{\H_3(\O_{\mathbb{C}})}\bigg(\sum_{k=1}^n \mathbf{1}_{B_k}(\omega)a_k\bigg) \d\nu(\omega) 
=\int_\Omega \tau_{\H_3(\O_{\mathbb{C}})}(f(\omega)) \d\nu(\omega). \nonumber
\end{align}
Since $\H_3(\O_{\mathbb{C}})$ is finite-dimensional, every element of $\L^\infty(\Omega,\H_3(\O_{\mathbb{C}}))$ can be approximated in the essential supremum norm by simple $\H_3(\O_{\mathbb{C}})$-valued functions. Since $\tau$ and $\tau_{\H_3(\O_{\mathbb{C}})}$ are continuous, we deduce \eqref{eq-trace-disintegration-Albert}, i.e.,
\begin{equation}
\label{eq-trace-disintegration-Albert-bis}
\tau(f)
=
\int_\Omega \tau_{\H_3(\O_{\mathbb{C}})}(f(\omega)) \d\nu(\omega),
\quad f \in \L^\infty(\Omega,\H_3(\O_{\mathbb{C}})).
\end{equation}
The Jordan operations and the continuous functional calculus on $\L^\infty(\Omega,\H_3(\O_{\mathbb{C}}))$ are defined pointwise. Thus, for any $f \in \L^\infty(\Omega,\H_3(\O_{\mathbb{C}}))$ and almost every $\omega \in \Omega$,
\begin{equation}
\label{inter-39}
(f^* \circ f)(\omega)
=f(\omega)^* \circ f(\omega)
\end{equation}
and
\begin{equation}
\label{inter-40}
(f^*\circ f)^{\frac p2}(\omega)
=\big((f^*\circ f)(\omega)\big)^{\frac p2}
\ov{\eqref{inter-39}}{=}\big(f(\omega)^*\circ f(\omega)\big)^{\frac p2}.
\end{equation}
It follows from \eqref{eq-trace-disintegration-Albert} that
\begin{align*}
\MoveEqLeft
\norm{f}_{\L^p(\cal{M})}
\ov{\eqref{norm-Lp-intro}}{=} \Big(\tau\big[(f^*\circ f)^{\frac p2}\big]\Big)^{\frac{1}{p}}
\ov{\eqref{eq-trace-disintegration-Albert-bis}}{=} \bigg(\int_\Omega \tau_{\H_3(\O_{\mathbb{C}})}\big((f^*\circ f)^{\frac p2}(\omega)\big) \d\nu(\omega)\bigg)^{\frac{1}{p}}\\
&\ov{\eqref{inter-40}}{=} \bigg(\int_\Omega \tau_{\H_3(\O_{\mathbb{C}})}\left[ \big(f(\omega)^*\circ f(\omega)\big)^{\frac p2} \right] \d\nu(\omega)\bigg)^{\frac{1}{p}} 
\ov{\eqref{norm-Lp-intro}}{=}\bigg(\int_\Omega\norm{f(\omega)}_{\L^p(\H_3(\O_{\mathbb{C}}))}^p \d\nu(\omega)\bigg)^{\frac{1}{p}}. 
\end{align*}
This proves \eqref{eq-Lp-Bochner-Albert}.
\end{proof}

\begin{thm}
\label{thm-spectral-norm-Linfty-Albert}
Let $(\Omega,\Sigma,\mu)$ be a localizable measure space. Assume that $\cal{M}=\L^\infty(\Omega,\H_3(\O_{\mathbb{C}}))$ is equipped with a normal finite faithful trace $\tau$. Suppose that $1 \leq p < \infty$. Then \eqref{norm-Lp-intro} defines a norm on $\cal{M}=\L^\infty(\Omega,\H_3(\O_{\mathbb{C}}))$. Moreover, the completion $\L^p(\cal{M})$ is canonically isometric to the Bochner space $\L^p(\Omega,\nu,\L^p(\H_3(\O_{\mathbb{C}})))$.
\end{thm}

\begin{proof}
By Theorem \ref{thm-spectral-norm-complexified-Albert}, the space $\L^p(\H_3(\O_{\mathbb C}))$ is a finite-dimensional Banach space. By \eqref{eq-Lp-Bochner-Albert}, the canonical inclusion $
\L^\infty(\Omega,\H_3(\O_{\mathbb C})) \to \L^p\big(\Omega,\nu;\L^p(\H_3(\O_{\mathbb C}))\big)$ is isometric when the space on the left is equipped with $\norm{\cdot}_{\L^p(\cal M)}$. Since $\nu$ is finite, simple $\H_3(\O_{\mathbb C})$-valued functions belong to $\L^\infty(\Omega,\H_3(\O_{\mathbb C}))$ and are dense in the Bochner space $\L^p(\Omega,\nu;\L^p(\H_3(\O_{\mathbb C})))$ by \cite[Lemma 1.2.19 p.~23]{HvNVW16}. We conclude that the completion $\L^p(\cal M)$ is canonically isometric to this Bochner space.
\end{proof}

\section{Case of $\JBW^*$-algebras}
\label{sec-JBW-star}

Consider some $\JBW^*$-algebra $\cal{M}$ equipped with a normal finite faithful trace $\tau$. Assume that $1 \leq p < \infty$. In this section, we establish that the formula $\norm{x}_{\L^p(\cal{M})}
\ov{\mathrm{def}}{=} \big(\tau \big[(x^* \circ x)^{\frac{p}{2}}\big]\big)^{\frac{1}{p}}$, where $x \in \cal{M}$,  of \eqref{norm-Lp-intro} defines a norm on $\cal{M}$, by using a structure result on the $\JBW^*$-algebras.

\begin{thm}
\label{thm:JBW*-Lp-norm}
Let $\cal{M}$ be a $\JBW^*$-algebra equipped with a normal finite faithful trace $\tau$. Suppose that $1 \leq p < \infty$. The expression of \eqref{norm-Lp-intro} defines a norm on $\cal{M}$.
\end{thm}

\begin{proof}
Consider the $\JBW$-algebra $\A$ associated with the $\JBW^*$-algebra $\cal{M}$. We have $\cal{M}=\A+\i \A$. Note that by \cite[Theorem 3.9 p.~374]{Shu79} or \cite[Theorem 4.23 p.~112]{AlS03}, 
a $\JBW$-algebra $\A$ can be uniquely decomposed as a direct sum 
\begin{equation}
\label{decompo}
\A
=\A_{\mathrm{sp}} \oplus \A_{\textrm{exp}},
\end{equation}
where $\A_{\mathrm{sp}}$ is a $\JW$-algebra and $\A_{\textrm{exp}}$ is a purely exceptional $\JBW$-algebra, i.e., isomorphic to the real algebra $\L^\infty_\R(\Omega,\mu,\mathrm{H}_3(\O))$ for a localizable measure space $(\Omega,\mu)$.


The trace $\tau$ decomposes as a direct sum $\tau = \tau_{\mathrm{sp}} \oplus \tau_{\textrm{exp}}$ of normal finite faithful traces on the summands. We denote by $\cal{M}_{\mathrm{sp}}$ and $\cal{M}_{\textrm{exp}}$ the associated $\JW^*$-algebras. We have an identification $\cal{M}=\cal{M}_{\mathrm{sp}} \oplus \cal{M}_{\textrm{exp}}$. 
The formula defining $\norm{\cdot}_{\L^p(\cal{M})}$ is compatible with this direct sum decomposition, i.e., if $x = x_{\mathrm{sp}}+x_{\textrm{exp}}$ is an element of $\cal{M}$ then we have
\[
\norm{x}_{\L^p(\cal{M})}
\ov{\eqref{norm-Lp-intro}}{=} \big(\norm{x_{\mathrm{sp}}}_{\L^p(\cal{M}_{\mathrm{sp}})}^p+\norm{x_{\textrm{exp}}}_{\L^p(\cal{M}_{\textrm{exp}})}^p\big)^{\frac{1}{p}}.
\]
On the one hand, $\norm{\cdot}_{\L^p(\cal{M}_{\mathrm{sp}})}$ is a norm by \cite[Theorem 3.7]{ArL26}. On the other hand, $\norm{\cdot}_{\L^p(\cal{M}_{\textrm{exp}})}$ is a norm by Theorem \ref{thm-spectral-norm-Linfty-Albert}. Consequently, the $\ell^p$-sum of these norms is again a norm on the direct sum by \cite[Exercise 1.88 p.~68]{Meg98}. This proves the result.
\end{proof}

\section{Optimal comparison with interpolation nonassociative $\L^p$-spaces}
\label{sec-comparison-interpolation-Albert}

We finally compare the spectral spaces constructed previously with the nonassociative $\L^p$-spaces obtained by complex interpolation. For tracial $\JW^*$-algebras, \cite[Theorem 5.1]{ArL26} shows that the two constructions are equivalent with the optimal universal constant $2^{|\frac1p-\frac12|}$. It is not immediate that the same estimates should survive in the exceptional case, since $\H_3(\O_{\mathbb C})$ has no associative realization to which the comparison theorem of \cite{ArL26} could be applied globally.

%
%


The key observation is twofold. First, every unital $\JBW^*$-subalgebra of $\H_3(\O_{\mathbb C})$ is the range of a trace-preserving Jordan conditional expectation. By interpolation, this implies that the inclusion of such a subalgebra is isometric for the interpolation nonassociative $\L^p$-norm. Second, although $\H_3(\O_{\mathbb C})$ itself is exceptional, every individual
element belongs to a finite-dimensional special unital $\JBW^*$-subalgebra. Indeed, if $x=a+\i b$, the unital real Jordan algebra generated by $a$ and $b$ is special by the Shirshov--Cohn theorem. These two observations allow us to localize both the spectral and the interpolation norms to special subalgebras, where the comparison theorem of \cite{ArL26} applies. In particular, the exceptional nature of $\H_3(\O_{\mathbb C})$ is a global obstruction rather than an obstruction visible on the Jordan algebra generated by the real and imaginary parts of a single element.  
We will use the existence of trace-preserving Jordan conditional expectations
established in \cite[Proposition 2.5]{Arh24a}, together with their interpolation
properties from \cite[Proposition 3.11]{Arh24a}, as in the proof of \cite[Theorem 5.1]{ArL26}.

\begin{lemma}
\label{lemma-local-special-reduction-Albert}
Suppose that $1\leq p<\infty$.
\begin{enumerate}
\item
Let $\cal N$ be a unital $\JBW^*$-subalgebra of $\H_3(\O_{\mathbb C})$. Then, for any
$x \in \cal N$, we have
\begin{equation}
\label{eq-restriction-spectral-Albert}
\norm{x}_{\L^p(\H_3(\O_{\mathbb C}))}
=
\norm{x}_{\L^p(\cal N)}
\end{equation}
and
\begin{equation}
\label{eq-restriction-interpolation-Albert}
\norm{x}_{\L^{p,A}(\H_3(\O_{\mathbb C}))}
=
\norm{x}_{\L^{p,A}(\cal N)}.
\end{equation}
\item
For any $x \in \H_3(\O_{\mathbb C})$, there exists a finite-dimensional special unital
$\JBW^*$-subalgebra $\cal N_x$ of $\H_3(\O_{\mathbb C})$ containing $x$.
\end{enumerate}
\end{lemma}

\begin{proof}
We first prove the first assertion. Let $\cal N$ be a unital $\JBW^*$-subalgebra of $\H_3(\O_{\mathbb C})$. By \cite[Proposition 2.5]{Arh24a}, there exists a $\tau_{\H_3(\O_{\mathbb C})}$-preserving normal faithful Jordan conditional expectation $Q\co \H_3(\O_{\mathbb C}) \to \H_3(\O_{\mathbb C})$ on $\cal{N}$. Let $J \co \cal N \to \H_3(\O_{\mathbb C})$ denote the canonical inclusion. By \cite[Proposition 3.11]{Arh24a}, the map $Q$ induces a contractive projection $Q_p \co \L^{p,A}(\H_3(\O_{\mathbb C})) \to \L^{p,A}(\H_3(\O_{\mathbb C}))$ whose range, say $X_p$, is isometric to $\L^{p,A}(\cal N)$. More precisely, there exists an isometry
\[
I_p\co\L^{p,A}(\cal N)\to X_p
\subset \L^{p,A}(\H_3(\O_{\mathbb C}))
\]
such that $I_p|_{\cal N}=J$. Consequently, for any $x \in \cal N$, we have
\[
\norm{x}_{\L^{p,A}(\cal N)}
=
\norm{I_p(x)}_{\L^{p,A}(\H_3(\O_{\mathbb C}))}
=
\norm{J(x)}_{\L^{p,A}(\H_3(\O_{\mathbb C}))}
=
\norm{x}_{\L^{p,A}(\H_3(\O_{\mathbb C}))}.
\]
This proves \eqref{eq-restriction-interpolation-Albert}.

For the spectral norm, consider some $x \in \cal N$. Since $\cal N$ is a $\JBW^*$-subalgebra of $\H_3(\O_{\mathbb C})$, we have $x^*\circ x \in \cal N_+$, and the continuous functional calculus of $x^*\circ x$ is the same whether it is performed in $\cal N$ or in $\H_3(\O_{\mathbb C})$. Since the trace on $\cal N$ is the restriction of $\tau_{\H_3(\O_{\mathbb C})}$, we obtain
\begin{align*}
\norm{x}_{\L^p(\cal N)}
&\ov{\eqref{norm-Lp-intro}}{=}
\left(\tau_{\H_3(\O_{\mathbb C})}\left[(x^*\circ x)^{\frac p2}\right]\right)^{\frac{1}{p}}
\ov{\eqref{norm-Lp-intro}}{=}
\norm{x}_{\L^p(\H_3(\O_{\mathbb C}))}.
\end{align*}
This proves \eqref{eq-restriction-spectral-Albert}. Now, we prove the second assertion. Let $x \in \H_3(\O_{\mathbb C})$ and write $x=a+\i b$ for some $a,b \in \cal A$. Let $\cal B$ be the unital real Jordan subalgebra of $\cal A$ generated by $a$ and $b$. By the Shirshov--Cohn theorem \cite[Theorem 3.1.55 p.~337]{CGRP14}, the Jordan algebra $\cal B$ is special. Hence its complexification $
\cal N_x
\ov{\mathrm{def}}{=}\cal B+\i\cal B$ is a finite-dimensional special unital $\JBW^*$-subalgebra of $\H_3(\O_{\mathbb C})$ containing $x$.
\end{proof}

Now, we pass to the measurable exceptional part. Let $(\Omega,\Sigma,\mu)$ be a localizable measure space, let $\cal{M} \ov{\mathrm{def}}{=} \L^\infty(\Omega,\H_3(\O_{\mathbb C}))$, and suppose that $\cal{M}$ is equipped with a normal finite faithful trace $\tau$. By Proposition~\ref{prop-traces-on-Linfty}, there exists a finite measure $\nu$ equivalent to $\mu$ such that
\[
\tau(f)
=\int_\Omega\tau_{\H_3(\O_{\mathbb C})}(f(\omega))\d\nu(\omega),
\quad f \in \L^\infty(\Omega,\H_3(\O_{\mathbb C})).
\]

\begin{prop}
\label{prop-interpolation-Bochner-Albert}
Suppose that $1\leq p<\infty$. We have a canonical isometric identification
\begin{equation}
\label{eq-interpolation-Bochner-Albert}
\L^{p,A}(\cal M)
=\L^p\big(\Omega,\nu,\L^{p,A}(\H_3(\O_{\mathbb C}))\big).
\end{equation}
In particular, for any $f \in \L^\infty(\Omega,\H_3(\O_{\mathbb C}))$, we have
\[
\norm{f}_{\L^{p,A}(\cal{M})}
=
\left(\int_\Omega\norm{f(\omega)}_{\L^{p,A}(\H_3(\O_{\mathbb C}))}^p\d\nu(\omega)\right)^{\frac1p}.
\]
\end{prop}

\begin{proof}
Since $\H_3(\O_{\mathbb C})$ is finite-dimensional, we have the canonical isometric identifications $
\cal{M} =\L^\infty(\Omega,\nu,\H_3(\O_{\mathbb C}))$ and $\cal M_*=\L^1(\Omega,\nu,\H_3(\O_{\mathbb C})_*)$, and the trace embedding of $\cal{M}$ into $\cal{M}_*$ acts pointwise. Therefore, by the vector-valued complex interpolation theorem \cite[Theorem 2.2.6 p.~91]{HvNVW16}, we have isometrically
\begin{align*}
\MoveEqLeft
\L^{p,A}(\cal M)
=
\big(\L^\infty(\Omega,\nu,\H_3(\O_{\mathbb C})),\L^1(\Omega,\nu,\H_3(\O_{\mathbb C})_*)\big)_{\frac1p}
=
\L^p\big(\Omega,\nu,(\H_3(\O_{\mathbb C}),\H_3(\O_{\mathbb C})_*)_{\frac1p}\big) \\
&= \L^p(\Omega,\nu,\L^{p,A}(\H_3(\O_{\mathbb{C}}))).
\end{align*}
\end{proof}

\begin{thm}
\label{thm-optimal-comparison-Albert}
Let $\cal M=\L^\infty(\Omega,\H_3(\O_{\mathbb C}))$ be equipped with a normal finite faithful trace. Suppose that $1\leq p<\infty$.

\begin{enumerate}
\item If $1\leq p\leq2$, then for any $f\in\cal M$,
\begin{equation}
\label{eq-optimal-comparison-Albert-p-le-2}
\norm{f}_{\L^{p,A}(\cal M)}
\leq
\norm{f}_{\L^p(\cal M)}
\leq
2^{\frac1p-\frac12}\norm{f}_{\L^{p,A}(\cal M)}.
\end{equation}

\item If $2\leq p<\infty$, then for any $f \in \cal M$,
\begin{equation}
\label{eq-optimal-comparison-Albert-p-ge-2}
\norm{f}_{\L^p(\cal M)}
\leq
\norm{f}_{\L^{p,A}(\cal M)}
\leq
2^{\frac12-\frac1p}\norm{f}_{\L^p(\cal M)}.
\end{equation}
\end{enumerate}
The constants are optimal.
\end{thm}

\begin{proof}
Let $x \in \H_3(\O_{\mathbb C})$. By the second assertion of Lemma~\ref{lemma-local-special-reduction-Albert}, there exists a finite-dimensional special unital $\JBW^*$-subalgebra $\cal N_x$ of $\H_3(\O_{\mathbb C})$ containing $x$. By the first assertion of the same lemma, both norms of $x$ can be computed in $\cal N_x$, namely,
\[
\norm{x}_{\L^p(\H_3(\O_{\mathbb C}))}
=
\norm{x}_{\L^p(\cal N_x)}
\quad\text{and}\quad
\norm{x}_{\L^{p,A}(\H_3(\O_{\mathbb C}))}
=
\norm{x}_{\L^{p,A}(\cal N_x)}.
\]
Since $\cal N_x$ is special and finite-dimensional, it admits a realization as a $\JW^*$-algebra. Applying \cite[Theorem 5.1]{ArL26} to $\cal N_x$, we obtain
\[
\norm{x}_{\L^{p,A}(\H_3(\O_{\mathbb C}))}
\leq
\norm{x}_{\L^p(\H_3(\O_{\mathbb C}))}
\leq
2^{\frac1p-\frac12}\norm{x}_{\L^{p,A}(\H_3(\O_{\mathbb C}))}
\]
if $1\leq p\leq2$, and
\[
\norm{x}_{\L^p(\H_3(\O_{\mathbb C}))}
\leq
\norm{x}_{\L^{p,A}(\H_3(\O_{\mathbb C}))}
\leq
2^{\frac12-\frac1p}\norm{x}_{\L^p(\H_3(\O_{\mathbb C}))}
\]
if $2\leq p<\infty$.

The spectral Bochner formula of Theorem~\ref{thm-spectral-norm-Linfty-Albert} and Proposition~\ref{prop-interpolation-Bochner-Albert} allow us to integrate these pointwise inequalities, which gives \eqref{eq-optimal-comparison-Albert-p-le-2} and \eqref{eq-optimal-comparison-Albert-p-ge-2}. It remains to prove optimality. Choose an imaginary octonionic unit $j \in \O$ with $j^2=-1$ and set $\mathbb{C}_j \ov{\mathrm{def}}{=} \Span_\R \{1,j\}$. The complexification of the unital Jordan subalgebra
\[
\left\{
\begin{bmatrix}
\alpha&z&0\\
\overline z&\beta&0\\
0&0&\gamma
\end{bmatrix}
:
\alpha,\beta,\gamma\in\R,\ z\in\mathbb C_j
\right\}
\subset\H_3(\O)
\]
is a unital special $\JBW^*$-subalgebra $\cal N_0$ of $\H_3(\O_{\mathbb C})$ isomorphic to $\M_2(\mathbb C)\oplus\mathbb C$, and the restriction of $\tau_{\H_3(\O_{\mathbb C})}$ is given by
\[
\tau_{\H_3(\O_{\mathbb C})}(z,\lambda)=\frac13\big(\Tr(z)+\lambda\big).
\]
Consider the element $x_0 \ov{\mathrm{def}}{=}
\left(
\begin{bmatrix}
0&1\\
0&0
\end{bmatrix},0
\right)$ of $\cal N_0$. The computation of \cite[Remark 4.8]{ArL26}, applied to the first summand, gives
\[
\frac{\norm{x_0}_{\L^p(\cal N_0)}}{\norm{x_0}_{\L^{p,A}(\cal N_0)}}
=
2^{\frac1p-\frac12}.
\]
Since $\cal N_0$ is a unital $\JBW^*$-subalgebra of $\H_3(\O_{\mathbb C})$, the first assertion of Lemma~\ref{lemma-local-special-reduction-Albert} gives
\[
\norm{x_0}_{\L^p(\H_3(\O_{\mathbb C}))}
=
\norm{x_0}_{\L^p(\cal N_0)}
\quad\text{and}\quad
\norm{x_0}_{\L^{p,A}(\H_3(\O_{\mathbb C}))}
=
\norm{x_0}_{\L^{p,A}(\cal N_0)}.
\]
Hence
\begin{equation}
\label{eq-sharp-ratio-Albert}
\frac{\norm{x_0}_{\L^p(\H_3(\O_{\mathbb C}))}}
{\norm{x_0}_{\L^{p,A}(\H_3(\O_{\mathbb C}))}}
=
2^{\frac1p-\frac12}.
\end{equation}
The factor $3^{-\frac1p}$ coming from the normalization of the trace occurs in both norms and therefore cancels. If $1\leq p\leq2$, \eqref{eq-sharp-ratio-Albert} shows that the constant
$2^{\frac1p-\frac12}$ in \eqref{eq-optimal-comparison-Albert-p-le-2} is optimal. If $2\leq p<\infty$, taking the reciprocal in \eqref{eq-sharp-ratio-Albert} gives
\[
\frac{\norm{x_0}_{\L^{p,A}(\H_3(\O_{\mathbb C}))}}
{\norm{x_0}_{\L^p(\H_3(\O_{\mathbb C}))}}
=
2^{\frac12-\frac1p},
\]
which proves the optimality of the constant $2^{\frac12-\frac1p}$ in \eqref{eq-optimal-comparison-Albert-p-ge-2}. Finally, consider the vector-valued function $f \co \Omega \to \H_3(\O_{\mathbb C})$ defined by 
\[
f_0(\omega)
\ov{\mathrm{def}}{=}x_0,
\quad \omega\in\Omega.
\]
We have
\[
\norm{f_0}_{\L^p(\cal M)}
=
\nu(\Omega)^{\frac1p}\norm{x_0}_{\L^p(\H_3(\O_{\mathbb C}))}
\quad \text{and} \quad
\norm{f_0}_{\L^{p,A}(\cal M)}
=
\nu(\Omega)^{\frac1p}\norm{x_0}_{\L^{p,A}(\H_3(\O_{\mathbb C}))}.
\]
Thus the same ratios are attained in $\cal M$. The constants equal to $1$ are optimal since the spectral and interpolation norms coincide on selfadjoint elements.
\end{proof}

The preceding theorem, together with the corresponding result for $\JW^*$-algebras, yields the comparison for arbitrary $\JBW^*$-algebras.

\begin{cor}
\label{cor-optimal-comparison-general-JBW}
Let $\cal M$ be a $\JBW^*$-algebra equipped with a normal finite faithful trace $\tau$, and suppose that $1\leq p<\infty$.

\begin{enumerate}
\item If $1\leq p\leq2$, then for any $x \in \cal M$,
\begin{equation}
\label{eq-optimal-comparison-general-p-le-2}
\norm{x}_{\L^{p,A}(\cal M)}
\leq
\norm{x}_{\L^p(\cal M)}
\leq
2^{\frac1p-\frac12}\norm{x}_{\L^{p,A}(\cal M)}.
\end{equation}

\item If $2\leq p<\infty$, then for any $x \in \cal M$,
\begin{equation}
\label{eq-optimal-comparison-general-p-ge-2}
\norm{x}_{\L^p(\cal M)}
\leq
\norm{x}_{\L^{p,A}(\cal M)}
\leq
2^{\frac12-\frac1p}\norm{x}_{\L^p(\cal M)}.
\end{equation}
\end{enumerate}

The constants $2^{|\frac1p-\frac12|}$ are optimal as universal constants over the class of tracial $\JBW^*$-algebras.
\end{cor}

\begin{proof}
Let $\cal A=\cal M_{\sa}$. By the structure theorem \cite[Theorem 3.9 p.~374]{Shu79} \cite[Theorem 4.23 p.~112]{AlS03} for $\JBW$-algebras, we have a decomposition $\cal{A}=\cal A_{\mathrm{sp}}\oplus\cal A_{\mathrm{exp}}$, where $\cal A_{\mathrm{sp}}$ is a $\JW$-algebra and $\cal A_{\exp}$ is purely exceptional, hence isomorphic to an algebra of the form $\L^\infty_\R(\Omega,\H_3(\O))$ for a localizable measure space $\Omega$. After complexification, we have $\cal{M} = \cal M_{\mathrm{sp}} \oplus \cal M_{\exp}$, where $\cal M_{\mathrm{sp}}$ is a $\JW^*$-algebra and $\cal M_{\mathrm{exp}}$ is isomorphic to $\L^\infty(\Omega,\H_3(\O_{\mathbb C}))$. The trace decomposes as $\tau=\tau_{\mathrm{sp}} \oplus \tau_{\mathrm{exp}}$. If $x=x_{\mathrm{sp}}+x_{\mathrm{exp}}$, then
\[
\norm{x}_{\L^p(\cal M)}^p
=
\norm{x_{\mathrm{sp}}}_{\L^p(\cal M_{\mathrm{sp}})}^p
+
\norm{x_{\mathrm{exp}}}_{\L^p(\cal M_{\mathrm{exp}})}^p.
\]
Moreover, since $\cal{M} = \cal{M}_{\mathrm{sp}} \oplus_\infty \cal{M}_{\mathrm{exp}}$ and $\cal M_*=(\cal{M}_{\mathrm{sp}})_*\oplus_1 (\cal{M}_{\mathrm{exp}})_*$, the standard interpolation formula for direct sums gives
\[
\L^{p,A}(\cal{M})
=
\L^{p,A}(\cal{M}_{\mathrm{sp}}) \oplus_p \L^{p,A}(\cal{M}_{\mathrm{exp}})
\]
isometrically. Hence
\[
\norm{x}_{\L^{p,A}(\cal M)}^p
=
\norm{x_{\mathrm{sp}}}_{\L^{p,A}(\cal{M}_{\mathrm{sp}})}^p
+
\norm{x_{\mathrm{exp}}}_{\L^{p,A}(\cal{M}_{\mathrm{exp}})}^p.
\]
Applying \cite[Theorem 5.1]{ArL26} to the special summand and Theorem~\ref{thm-optimal-comparison-Albert} to the exceptional summand, and then taking the $\ell^p$-sum, gives \eqref{eq-optimal-comparison-general-p-le-2} and \eqref{eq-optimal-comparison-general-p-ge-2}. The universal constants cannot be improved, since they are already optimal on the matrix algebra $\M_2(\mathbb{C})$, and also on the exceptional algebra $\H_3(\O_{\mathbb{C}})$ by Theorem~\ref{thm-optimal-comparison-Albert}.
\end{proof}


\paragraph{Competing interests} The author declares that they have no competing interests.

\paragraph{Data availability} No data sets were generated during this study.

{\footnotesize

\vspace{0.2cm}

\noindent C\'edric Arhancet\\ 
\noindent 6 rue Didier Daurat, 81000 Albi, France\\
URL: \href{http://sites.google.com/site/cedricarhancet}{https://sites.google.com/site/cedricarhancet}\\
cedric.arhancet@protonmail.com\\
ORCID: 0000-0002-5179-6972 

}

\end{document}